\documentclass[a4paper, 12pt]{article}
\usepackage[margin=25mm]{geometry}

\usepackage{amsmath}
\usepackage{amsfonts}
\usepackage{amssymb}
\usepackage{amsthm}

\usepackage{graphicx}

\usepackage{subfigure} 

\theoremstyle{plain}
\newtheorem{theorem}{Theorem}[section]

\newtheorem{lem}[theorem]{Lemma}

\newtheorem{prop}[theorem]{Proposition}

\newtheorem{mtheorem}{Theorem} %

\theoremstyle{definition}

\newtheorem{rem}[theorem]{Remark}
\newtheorem{mrem}{Remark} %

\providecommand{\keywords}[1]
{
  \small	
  \textbf{\textit{Keywords.}} #1
}

\providecommand{\msc}[1]
{
  \small	
  \textbf{\textit{2020 Mathematics Subject Classification.}} #1
}

\title{Diffusive logistic equation with a non Lipschitz nonlinear boundary condition arising from coastal fishery harvesting: the resonant case\footnote{This work was supported by JSPS KAKENHI Grant Numbers JP18K03353 and JP23K03162.}}

\author{Kenichiro Umezu  \\
        \small Department of Mathematics, Faculty of Education, Ibaraki University, Mito 310--8512, Japan \\
        \small E-mail: \texttt{kenichiro.umezu.math@vc.ibaraki.ac.jp}  \\
}

\numberwithin{equation}{section}
\numberwithin{theorem}{section}

\allowdisplaybreaks[4] 

\date{} 
\begin{document} 
\maketitle

\begin{abstract}
For bifurcation analysis, we study the positive solution set for a semilinear elliptic equation of the logistic type, equipped with a sublinear boundary condition modeling coastal fishery harvesting. This work is a continuation of the author's previous studies  \cite{Um2023, Um2024}, where certain results were obtained in a {\it non resonant} case, including the existence, uniqueness, multiplicity, and strong positivity for positive solutions. 
In this paper, we consider the delicate {\it resonant} case and develop a sort of {\it non standard} bifurcation technique at zero to evaluate the positive solution set depending on a parameter. The nonlinear boundary condition is not right--differentiable at zero. 
\end{abstract} \hspace{10pt}


\msc{35J65, 35B32, 35J25, 92D40.}



\section{Introduction} 
In this paper, we consider positive solutions for the following semilinear elliptic equation of logistic type, equipped with the nonlinear boundary condition modeling coastal fishery harvested population. 
\begin{align} \label{p}
\begin{cases}
-\Delta u = u-u^{p} & \mbox{ in } \Omega, \\
\frac{\partial u}{\partial \nu}  = -\lambda u^q & \mbox{ on } \partial\Omega. 
\end{cases} 
\end{align}
Here, $\Omega\subset \mathbb{R}^N$, $N\geq1$, is a bounded domain with a smooth boundary $\partial\Omega$, 
$\Delta = \sum_{i=1}^N \frac{\partial^2}{\partial x_i^2}$ is the usual Laplacian in $\mathbb{R}^N$, 
$p$ and $q$ are given exponents satisfying $0<q<1<p$, 
$\lambda \geq 0$ is a parameter, and $\nu$ is the unit outer normal to $\partial\Omega$. 

For $p=2$, the unknown function $u>0$ in $\Omega$ ecologically represents the biomass of fish inhabiting a \textit{lake} $\Omega$, which obeys the logistic law \cite{CCbook}. Then, the sublinear boundary condition with $0<q<1$ represents fishery harvesting with the harvesting effort $\lambda$ on the \textit{lake coast} $\partial\Omega$, obeying the Cobb--Douglas production \cite[Subsection 2.1]{GUU19}.

Unless stated otherwise, we assume that $p<\frac{N+2}{N-2}$ if $N>2$. A nonnegative function $u\in H^1(\Omega)$ is called a \textit{nonnegative (weak) solution} of \eqref{p} if $u$ satisfies that  
\begin{align*} 
\int_{\Omega} \biggl( \nabla u \nabla \varphi -u\varphi + u^p \varphi \biggr) + \lambda \int_{\partial\Omega} u^{q} \varphi=0, \quad \varphi \in H^1(\Omega). 
\end{align*}
We may consider $(\lambda,u)$ as a nonnegative solution of \eqref{p}. 
It is clear that problem \eqref{p} has a solution $(\lambda,0)$ for every $\lambda\geq 0$, called a {\it trivial solution}, and therefore, the sets $\{ (\lambda,0) : \lambda\geq0\}$ 
and $\{ (\lambda, 0) : \lambda > 0 \}$ are referred to as {\it trivial lines}. 
We know (\cite{Ro2005}) that a nonnegative solution $u$ of \eqref{p} belongs to $W^{1,r}(\Omega)$ for $r>N$ (consequently, to $C^{\theta}(\overline{\Omega})$ for  $\theta \in (0,1)$). 
In fact, a nontrivial nonnegative solution $u$ of \eqref{p} satisfies that $u\in C^{2+\theta}(\Omega)$ for $\theta \in (0,1)$ and $u>0$ in $\Omega$ (\cite{GT83}, \cite{PW67}), which is called \textit{a positive solution}. In addition, if a positive solution $u$ of \eqref{p} satisfies $u>0$ in $\overline{\Omega}$, then $u\in C^{2+\theta}(\overline{\Omega})$ by the bootstrap argument using elliptic regularity. Thus, $u$ satisfies \eqref{p} in $\overline{\Omega}$ in the classical sense. 
It should be emphasized that we do not know if $u>0$ on the entirety of $\partial\Omega$ for a positive solution $u$ of \eqref{p}. In fact, Hopf's boundary point lemma (\cite{PW67}) does not work because of the lack of a {\it one-sided Lipschitz condition} \cite[(4.1.19)]{Pa92} for mapping $0\leq u \mapsto (-u^q)$. 

Then, we predict the structure of the positive solution set $\{ (\lambda, u) \}$ for \eqref{p}. Problem \eqref{p} is the Neumann problem when considering $\lambda = 0$, and $(\lambda, u)=(0,1)$ is a unique positive solution of \eqref{p} for $\lambda=0$; oppositely, this problem is regarded formally as the Dirichlet problem when considering $\lambda \to \infty$. For the positive solutions of the Dirichlet problem, let us introduce $\beta_\Omega>0$ as the smallest eigenvalue of the Dirichlet eigenvalue problem
\begin{align*} 
\begin{cases}
-\Delta \phi = \beta \phi & \mbox{ in } \Omega, \\
\phi = 0 & \mbox{ on } \partial \Omega. 
\end{cases}    
\end{align*} 
It is well known that $\beta_{\Omega}$ is simple with the positive eigenfunction $\phi_{\Omega}\in C^{2+\theta}(\overline{\Omega}) \cap H^1_0(\Omega)$, $0<\theta<1$, normalized as $\| \phi_\Omega\|=1$. 
The positivity means that $\phi_\Omega > 0$ in $\Omega$ and 
\begin{align} \label{bdry1513}
0< \min_{\partial\Omega}\biggl( -\frac{\partial \phi_{\Omega}}{\partial \nu} \biggr) \leq \max_{\partial\Omega}\biggl( -\frac{\partial \phi_{\Omega}}{\partial \nu} \biggr). 
\end{align}
As a matter of fact, $\beta_{\Omega}$ is characterized by the variational formula
\begin{align} \label{lamOcha}
\beta_{\Omega} = \inf\left\{ \int_{\Omega}|\nabla \phi|^2 : \phi\in H^1_0(\Omega), \ \int_{\Omega}\phi^2=1 \right\}. 
\end{align}
If $\beta_\Omega < 1$, then we denote by $u_{\mathcal{D}}\in C^{2+\theta}(\overline{\Omega})\cap H^{1}_{0}(\Omega)$, $0<\theta < 1$,  
the unique positive solution of the Dirichlet problem (\cite{BO86}) 
\begin{align} \label{Dp}
\begin{cases}
-\Delta u = u-u^{p} & \mbox{ in } \Omega, \\
u = 0 & \mbox{ on } \partial\Omega. 
\end{cases} 
\end{align}
Since the positive solution $(\lambda,u)=(0,1)$ is non degenerate, or asymptotically stable, 
the implicit function theorem allows us to deduce that there emanates from $(\lambda,u)=(0,1)$ a positive solution $U_{1}=U_{1,\lambda}$ of \eqref{p} for a sufficiently small $\lambda > 0$, see Remark \ref{rem1030}(i). 
Thus, when $\beta_{\Omega}<1$, it would be expected that $(\lambda, U_{1})$ is extended entirely to $\lambda>0$ and behaves asymptotically like the positive solution $u_{\mathcal{D}}$ of \eqref{Dp} as $\lambda \to \infty$. 
If $\beta_{\Omega}\geq1$, then problem \eqref{Dp} does not have any positive solution, and $u=0$ is the only nonnegative solution of \eqref{Dp} (\cite{BO86}); thus, the following two alternatives for \eqref{p} would be presented: $(\lambda, U_{1})$ vanishes as $\lambda \to \infty$ (if it exists in the entirety of $\lambda > 0$), or there is no positive solution of \eqref{p} for any sufficiently large $\lambda>0$.  

The case where $\beta_\Omega \neq 1$, referred to as {\it non resonance}, has been studied in the author's previous works \cite{Um2023, Um2024}. Then, we summarize the results obtained there.  
\setcounter{mtheorem}{-1}
\begin{mtheorem} \label{t0}
{\rm (I)} $u<1$ in $\overline{\Omega}$ and $u>0$ on $\Gamma$ with $|\Gamma|>0$ for some $\Gamma \subset \partial\Omega$ for a positive solution $u$ of \eqref{p}, provided that 
$\lambda>0$.  

{\rm (II)} Assume that $\beta_\Omega < 1$. Then, $u>0$ in $\overline{\Omega}$ for a positive solution $u$ of \eqref{p}.  
Conversely, problem \eqref{p} has at least one positive solution $U_{1}>0$ in $\overline{\Omega}$ for each $\lambda > 0$. Moreover, the following two assertions hold (Figure \ref{fig_g2}):
\begin{enumerate} \setlength{\itemsep}{3pt} 

\item $U_{1}=U_{1,\lambda}$ is {\rm unique} for $\lambda>0$ close to $0$ and for a sufficiently large $\lambda>0$. Additionally, $U_{1,\lambda}$ decreases, 
and it satisfies that $U_{1,\lambda} \longrightarrow 1$ in $C^{2+\theta}(\overline{\Omega})$, $0<\theta < 1$, as $\lambda \to 0^{+}$, while  
\begin{align}
U_{1,\lambda} \longrightarrow u_{\mathcal{D}} \ \mbox{ in } \ H^1(\Omega) \ 
\mbox{ as } \  \lambda \rightarrow \infty. \label{convtouD}
\end{align}
    
\item For $\Gamma_0 = \{ (\lambda, 0) : \lambda\geq 0\}\cup \{ (0, u) : u\geq0 \}$, the positive solution set $\{ (\lambda,u)\}$ 
does not meet any point on $\Gamma_0 \setminus \{ (0,1)\}$ in the topology of $[0,\infty)\times C(\overline{\Omega})$, that is, for $(\lambda, u)\in \Gamma_0\setminus \{ (0,1)\}$ there are no positive solutions of \eqref{p} in a neighborhood of $(\lambda,u)$. 

\end{enumerate}

\vspace{3pt}

{\rm (III)} Assume that $\beta_\Omega > 1$. Then, problem \eqref{p} has at least {\rm two} positive solutions $U_{1}$ and $U_{2}$ for $\lambda>0$ close to $0$ such that $U_{2}<U_{1}$ in $\overline{\Omega}$ 
and no positive solutions for $\lambda\geq \overline{\lambda}$ with some $\overline{\lambda}>0$. Moreover, the following two assertions hold. 
\begin{enumerate}\setlength{\itemsep}{3pt} 

\item The positive solution set $\{ (\lambda,u)\}$ does not meet any point on $\Gamma_0 \setminus \{ (0,0), (0,1)\}$ in the topology of $[0, \infty)\times C(\overline{\Omega})$.  

\item There exists a {\rm subcontinuum} (i.e., nonempty, closed, and connected subset) $\mathcal{C}^{\ast}_{p,q}=\{ (\lambda, u) \}$ in $[0, \infty)\times C(\overline{\Omega})$ for the nonnegative solutions of \eqref{p}, joining $(0,0)$ to $(0,1)$, see Figure \ref{figless1}.  

\end{enumerate}
\end{mtheorem}

\setcounter{mrem}{-1}
\begin{mrem} 

\label{mrem1}

(i) Assertion (I) holds unconditionally in the value of $\beta_{\Omega}>0$.  

(ii) The existence part in assertion (II) remains valid for any $p>1$, and it can be improved to the following result: problem \eqref{p} has a {\rm minimal} positive solution $\underline{U}_{1}$ and a {\rm maximal} positive solution $\overline{U}_{1}$ for each $\lambda > 0$ such that $0<\underline{U}_{1}\leq \overline{U}_{1}<1$ in $\overline{\Omega}$, meaning that $\underline{U}_{1} \leq U \leq \overline{U}_{1}$ in $\overline{\Omega}$ for a positive solution $U>0$ in $\overline{\Omega}$ of \eqref{p}, see for the proof assertion (I) of Theorem \ref{mt}. 

(iii) For assertion (III-ii) we find from assertion (III-i) that if $(\lambda,u) \in \mathcal{C}^{\ast}_{p,q}\setminus \{ (0,0), (0,1)\}$, then $\lambda > 0$, and $u$ is a positive solution of \eqref{p} for $\lambda$.      

(iv) As a similar study, a diffusive logistic population model with a fishery harvest inside a domain was analyzed by Cui, Li, Mei, and Shi \cite{CLMS2017}. 

\end{mrem}


Then, the objective of this paper is to investigate the structure of the positive solutions set $\{ (\lambda, u) : \lambda\geq0 \}$ for \eqref{p} in the case where $\beta_\Omega = 1$, referred to as {\it resonance}. We are now ready to present our main result for the resonant case. 
%
%
\begin{theorem} \label{mt}
Assume that $\beta_\Omega = 1$. Then, the following assertions hold. 

{\rm (I)} For $pq>1$, problem \eqref{p} has a {\rm minimal} positive solution $\underline{U}_{1}$ and a {\rm maximal} positive solution $\overline{U}_{1}$ for each $\lambda > 0$ such that $0<\underline{U}_{1}\leq \overline{U}_{1}<1$ in $\overline{\Omega}$, meaning that $\underline{U}_{1} \leq U \leq \overline{U}_{1}$ in $\overline{\Omega}$ for a positive solution $U>0$ in $\overline{\Omega}$ of \eqref{p}.  
Additionally, the following {\rm three} assertions hold (see Figure \ref{fig_g3}):
\begin{enumerate} \setlength{\itemsep}{3pt} 

\item $U_{1}:=\underline{U}_{1}=\overline{U}_{1}$ is {\rm unique} for $\lambda>0$ close to $0$.  
    
\item The positive solution set $\{ (\lambda,u)\}$ does not meet any point on $\Gamma_0 \setminus \{ (0,1)\}$ in the topology of $[0,\infty)\times C(\overline{\Omega})$, i.e., for $(\lambda, u)\in \Gamma_0\setminus \{ (0,1)\}$ there are no positive solutions of \eqref{p} in a neighborhood of $(\lambda,u)$. 

\item If $(\lambda_n,u_n)$ is a positive solution of \eqref{p} with $\lambda_n\to \infty$, then $u_n \rightarrow 0$ in $H^1(\Omega)$, and $\frac{u_n}{\| u_n \|} \rightarrow \phi_\Omega$ in $H^1(\Omega)$.
   
\end{enumerate}

\vspace{3pt}

{\rm (II)} For $pq<1$, problem \eqref{p} has at least {\rm two} positive solutions $U_1$ and $U_2$ for $\lambda>0$ close to $0$ such that $U_2<U_1$ in $\overline{\Omega}$, and there are no positive solutions of \eqref{p} for $\lambda\geq \overline{\lambda}$ with some $\overline{\lambda}>0$. Additionally, the following {\rm two} assertions hold. 
\begin{enumerate}\setlength{\itemsep}{3pt} 

\item The positive solution set does not meet any point on $\Gamma_0 \setminus \{ (0,0), (0,1)\}$ in the topology of $[0, \infty)\times C(\overline{\Omega})$. 

\item There exists a subcontinuum $\mathcal{C}^{\ast}_{p,q} =\{ (\lambda, u)\}$ in $[0, \infty)\times C(\overline{\Omega})$ for the nonnegative solutions of \eqref{p}, 
joining $(0,0)$ to $(0,1)$ and meeting the condition that if $(\lambda, u) \in \mathcal{C}_{p,q}^{\ast}\setminus \{ (0,0), (0,1)\}$, then $\lambda>0$ and 
$u$ is a positive solution of \eqref{p} for $\lambda$, see Figure \ref{figless1}. 

\end{enumerate}

\vspace{3pt} 

{\rm (III)} For $pq=1$, problem \eqref{p} has a {\rm unique} positive solution $U_1$ for $\lambda > 0$ close to $0$, which satisfies that $U_1>0$ in $\overline{\Omega}$, and there are no positive solutions for $\lambda\geq \overline{\lambda}$ with some $\overline{\lambda}>0$. 
Additionally, the following {\rm two} assertions hold. 

\begin{enumerate}\setlength{\itemsep}{3pt} 

\item For some $\underline{\lambda}>0$, the positive solution set $\{ (\lambda,u)\}$ does not meet any point on $\{ (\lambda, 0) : 0\leq \lambda < \underline{\lambda} \}\cup \{ (0,u) : u\geq0\}$ in the topology of $[0, \infty)\times C(\overline{\Omega})$ except for $(0,1)$.  

\item There exists a subcontinuum $\mathcal{C}^{\ast}_{p,q} =\{ (\lambda, u) \}$ 
in $[0, \infty)\times C(\overline{\Omega})$ for the nonnegative solutions of \eqref{p}, joining $(0,0)$ to $(0,1)$ and satisfying 

\begin{itemize} \setlength{\itemsep}{3pt} 

\item $\lambda>0$ if $(\lambda, u) \in \mathcal{C}_{p,q}^{\ast}\setminus \{ (0,0), (0,1)\}$; 

\item $\mathcal{C}_{p,q}^{\ast}$ includes the line segment $\{ (\lambda, 0) : \lambda \in [0,\underline{\lambda}) \}$ of trivial solutions, and there is not any other nonnegative solution of \eqref{p} in a neighborhood of $(\lambda,0)$ for $\lambda \in [0, \underline{\lambda})$.  

\end{itemize}
\end{enumerate}
\end{theorem}

\begin{figure}
   \begin{center}
   \subfigure[Case $\beta_\Omega <1$.]{
  \includegraphics[scale=0.22]{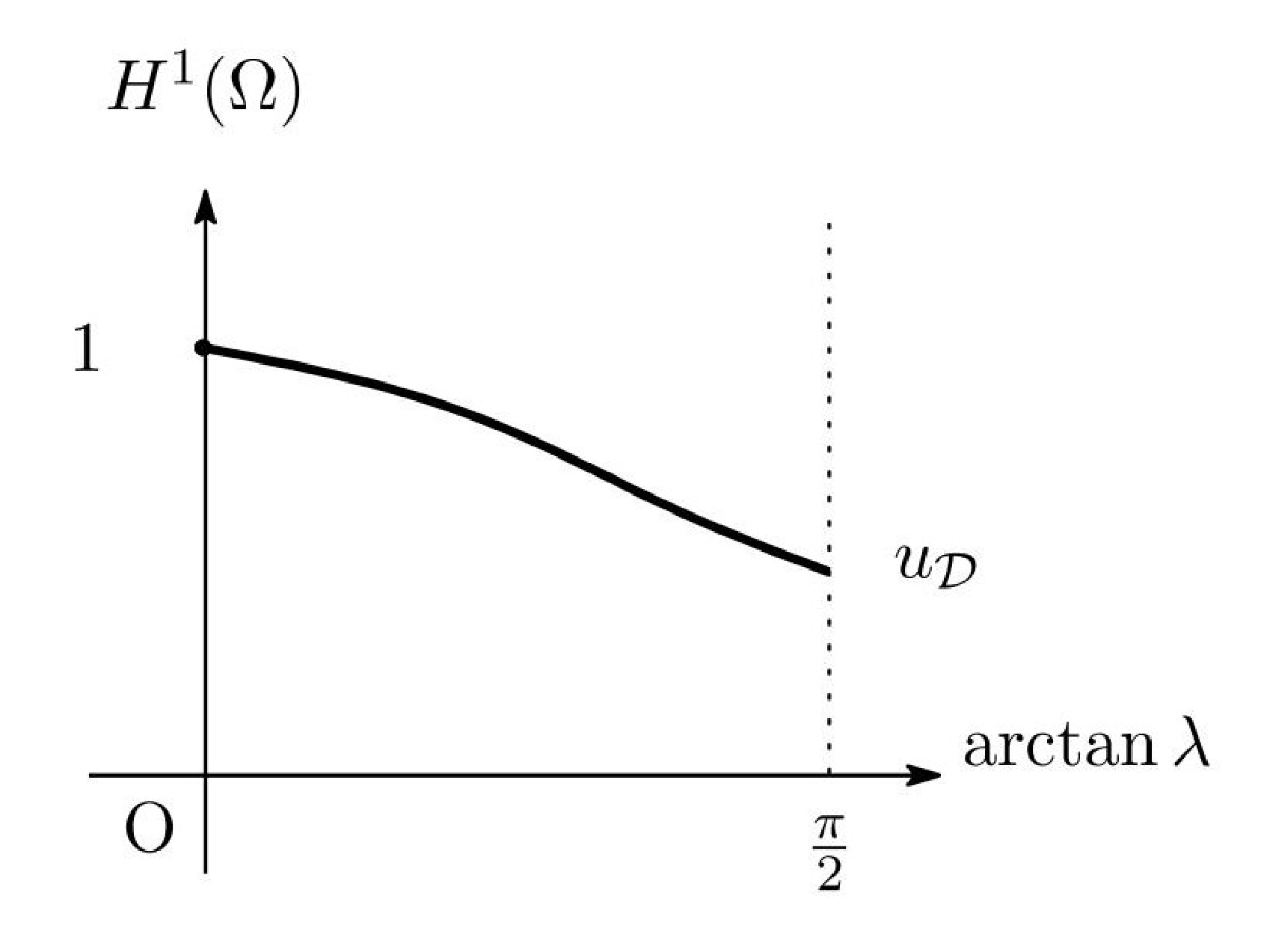}
     \label{fig_g2}
   }
   \hspace{0.5cm}
   \subfigure[Case $\beta_\Omega = 1$ and $pq>1$.]{
  \includegraphics[scale=0.22]{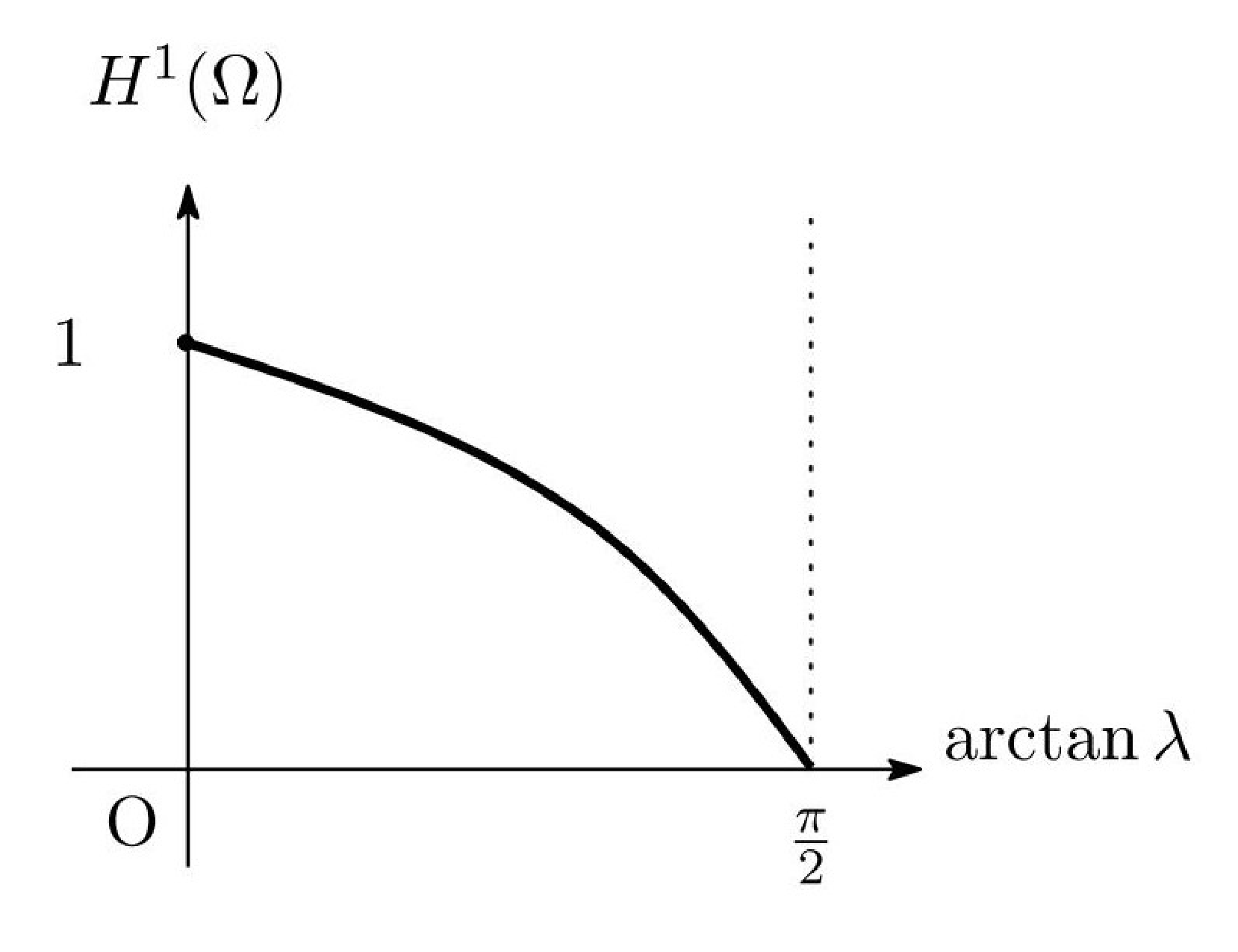}
     \label{fig_g3}
   }
   \end{center}
   \vspace{-0.5cm}
   \caption{Case with either $\beta_\Omega < 1$ or $\beta_\Omega =1$ and $pq>1$.}
   \label{fig_g23}
 \end{figure}


\begin{figure}
   \begin{center}
      \subfigure[Case $\beta_\Omega = 1$ and $pq=1$.]{
  \includegraphics[scale=0.22]{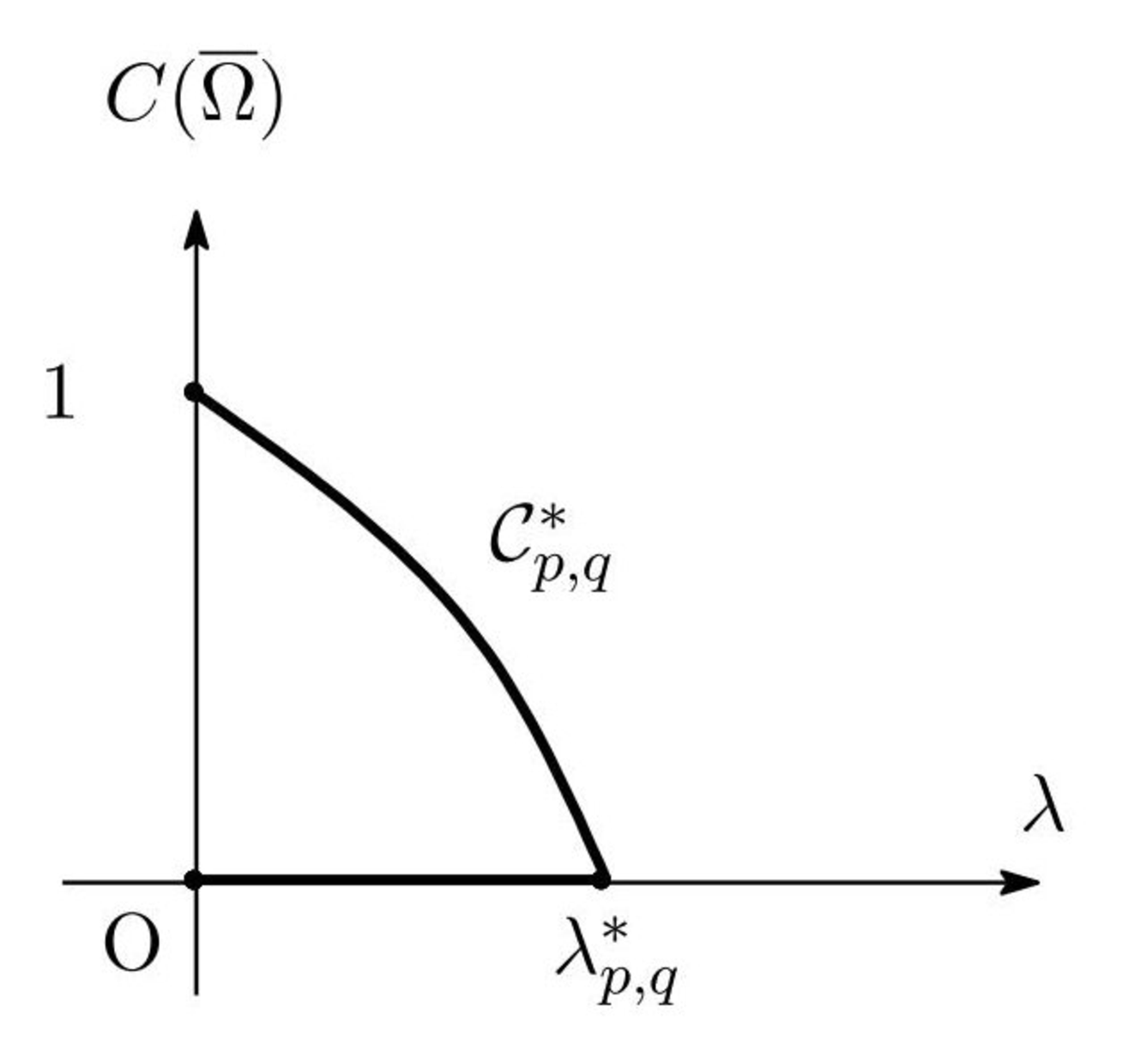}
     \label{figeq1}
   }
   \hspace{0.5cm}
   \subfigure[Case either $\beta_\Omega =1$ and $pq<1$ or $\beta_\Omega>1$.]{
  \includegraphics[scale=0.22]{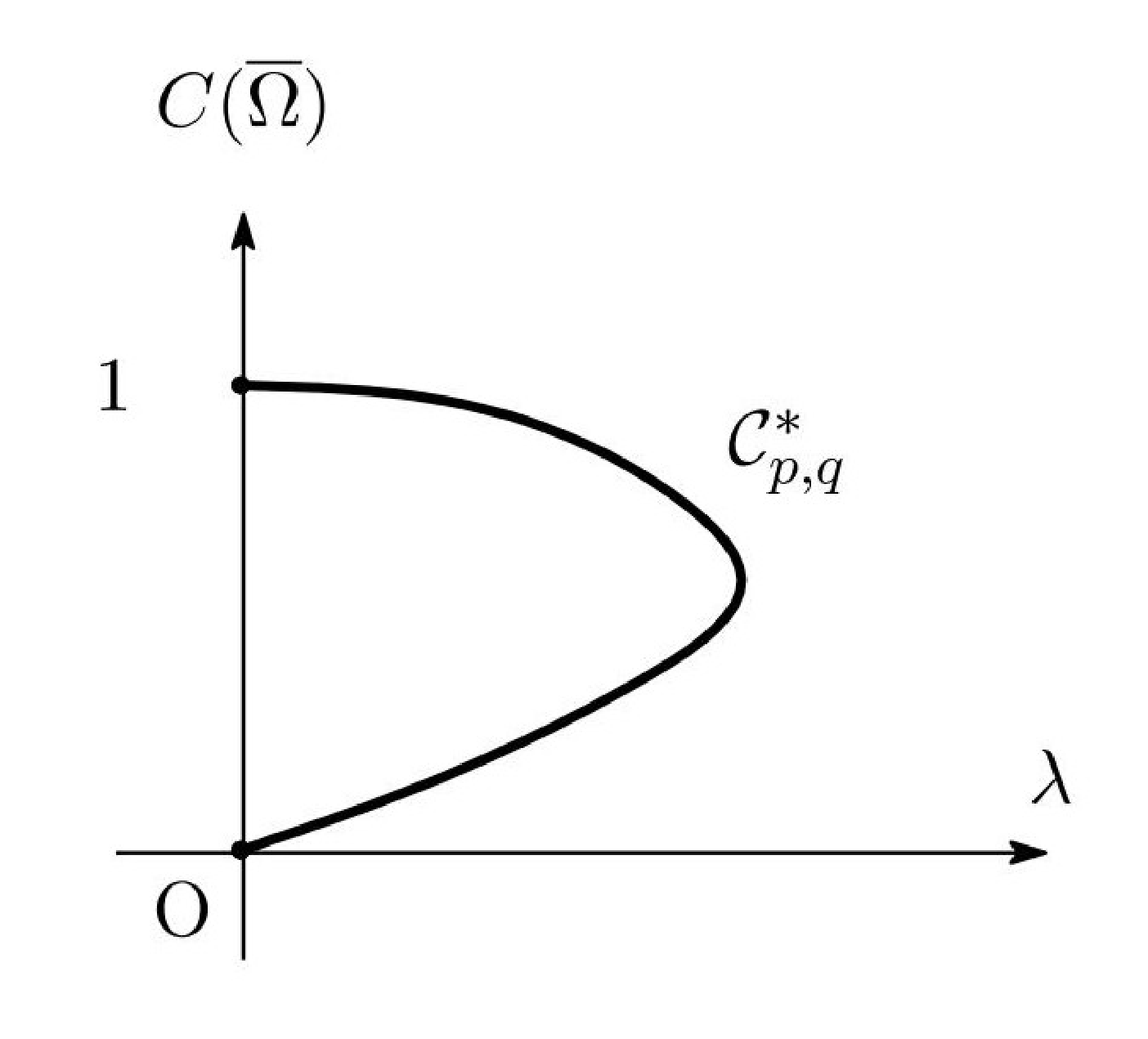}
     \label{figless1}
   }
   \end{center}
   \vspace{-0.5cm}
   \caption{Case with either $\beta_\Omega=1$ and $pq\leq 1$ or $\beta_\Omega > 1$.}
   \label{figcontinua}
 \end{figure}

\begin{rem} 

\label{rem1030}

(i) The positive solution $U_1>0$ in $\overline{\Omega}$ presented in assertions (I) to 
(III) is parameterized by $U_1=U_{1, \lambda}$ around $(\lambda, u)=(0,1)$ in $[0, \infty)\times C^{2+\theta}(\overline{\Omega})$, $0<\theta < 1$. This is a result of applying the implicit function theorem at $(\lambda, u)=(0,1)$. For the details, we refer to \cite[Theorem 1.1]{Um2023}. Thus, $U_{1,\lambda}$ is asymptotically stable, and $U_{1,\lambda} \rightarrow 1$ in $C^{2+\theta}(\overline{\Omega})$ as $\lambda \to 0^{+}$. Moreover, as in assertion (II-i) of Theorem \ref{t0}, $U_{1,\lambda}$ is decreasing for $\lambda$ if $\beta_{\Omega} = 1$ and $pq>1$. 

(ii) It is worth mentioning that there exist bifurcation points on $\{ (\lambda,0) : \lambda>0\}$ for positive solutions of \eqref{p} for the case when $\beta_\Omega = 1$ and $pq=1$. From assertion (III-i) we define  
\begin{align*}
    &\lambda^{\ast}_{p,q} = \sup\{ \underline{\lambda}>0 : \\ 
    & \hspace*{1.9cm} \mbox{there is no bifurcating positive solution of \eqref{p} at $(\lambda,0)$ for $0\leq \lambda <\underline{\lambda}$} \}. 
\end{align*}
Then, $\lambda^{\ast}_{p,q}$ is {\it finite} due to the existence of $\overline{\lambda}$ and the subcontinuum $\mathcal{C}_{p,q}^{\ast}$ joining $(0,0)$ to $(0,1)$. 
By the definition, the positive solution of \eqref{p} lying on $\mathcal{C}^{\ast}_{p,q}$ emanates from $(\lambda^{\ast}_{p,q}, 0)$, but $\lambda^{\ast}_{p,q}$ is not yet provided directly, see Figure \ref{figeq1}. 
This type of bifurcation point has been scarcely seen in the literature. 
Providing the value $\lambda^{\ast}_{p,q}$ directly is an interesting {\it open question}. 
\end{rem}

Problem \eqref{p} possesses a sublinear nonlinearity at infinity, as well as a concave--convex nature. 
For nonlinear elliptic problems with a concave--convex nature, we refer to \cite{ABC94, Ta95, Al99, ACP01, FGU03, FGU06, Ko13}. 
The sublinear nonlinearity $(-u^q)$ induced by the nonlinear boundary condition of \eqref{p} causes an absorption effect on $\partial\Omega$. 
Sublinear boundary conditions of the $u^q$ type were explored in 
\cite{GA04, imcom2007, GM08, RQU2014, RQU2015}; see \cite{GA04, GM08} for incoming flux on $\partial\Omega$, \cite{imcom2007} for a mixed case of absorption and incoming flux on $\partial\Omega$, and \cite{RQU2014, RQU2015} for an absorption case. 
 
The remainder of this paper is organized as follows. 
For the existence of the subcontinua of nonnegative solutions presented in assertions (II) and (III) of Theorem \ref{mt}, 
we face serious difficulties due to the lack of regularity around $u=0$ for $0<q<1$ and the resonant situation for $\beta_\Omega = 1$. 
It is difficult to apply the standard bifurcation theory from simple eigenvalues \cite{CR71, Ra71} directly. As a matter of fact, the linearized eigenvalue problem for \eqref{p} at $(\lambda,0)$ 
can {\it not} be considered because $u \mapsto (-u^q)$ is not right--differentiable at $u=0$. To overcome these difficulties, we study positive solutions of the following {\it regularized} problem {\it near} the resonance, associated to \eqref{p}.
\begin{align}  \label{p-abz}
\begin{cases}
-\Delta u = \beta u-|u|^{p-1}u & \mbox{ in } \Omega, \\
\frac{\partial u}{\partial \nu}  = -\lambda (u+\alpha)^{q-1}u & \mbox{ on } \partial\Omega, 
\end{cases} 
\end{align}
where $\alpha>0$ and $0<\beta<1$ are constants; thanks to $\alpha>0$, we consider the linearized eigenvalue problem at $u\geq 0$ associated to \eqref{p-abz}, which can be formulated as 
\begin{align} \label{165331} 
\begin{cases}
-\Delta \varphi = \beta \varphi - p|u|^{p-1} \varphi & \mbox{ in } \Omega, \\
\frac{\partial \varphi}{\partial \nu}  = -\lambda \{ (q-1)(u+\alpha)^{q-2}u + (u+\alpha)^{q-1} \}
\varphi & \mbox{ on } \partial\Omega. 
\end{cases} 
\end{align} 
Substituting $u=0$ into \eqref{165331}, we obtain the linear eigenvalue problem 
\begin{align} \label{epromu0}
\begin{cases}
-\Delta \varphi = \beta \varphi & \mbox{ in } \Omega, \\
\frac{\partial \varphi}{\partial \nu}  = -\lambda \alpha^{q-1} \varphi & \mbox{ on } \partial\Omega. 
\end{cases} 
\end{align} 
Then, thanks to $\beta < 1=\beta_{\Omega}$, we consider the smallest eigenvalue $\lambda_{\alpha, \beta}>0$ of \eqref{epromu0}, see \eqref{epromu}. Thus, we apply the standard theory for local and global bifurcation from simple eigenvalues to \eqref{p-abz} at $(\lambda_{\alpha, \beta},0)$ to deduce the existence of bifurcating positive solutions of \eqref{p-abz} from $\{ (\lambda,0)\}$ at this point. 
Then, considering $\alpha \to 0^{+}$, and then, considering $\beta \to 1^{-}$, we exploit the topological method proposed by Whyburn \cite[(9.12) Theorem]{Wh64} to conduct a limiting argument for the bifurcation result of \eqref{p-abz}. 
Such topological argument needs both an {\it a priori} upper bound of norm $\| u \|_{C(\overline{\Omega})}$ and that of parameter $\lambda > 0$ for the positive solutions $(\lambda,u)$ of \eqref{p-abz}, which are established in Sections \ref{sec:ener} and \ref{sec:bounds} by exploiting the energy approach prepared in Section \ref{sec:ener}. 
Section \ref{sec:asymp} is devoted to the study of the limiting behavior for the positive solutions $(\lambda, u)$ of \eqref{p-abz} as $\lambda \to 0^{+}$ or $\| u\|_{C(\overline{\Omega})} \rightarrow 0$, which supports our bifurcation theory. 
We then prove assertions (II) and (III) of Theorem \ref{mt} in Section \ref{sec:continuum}, and then, we prove assertion (I) of Theorem \ref{mt} in Section \ref{sec:sub} using the sub-- and supersolution method. 
Finally, in the appendix, we raise the question if assertion (I-iii) of Theorem \ref{mt} is consistent with \eqref{convtouD} and come up with it in a way using a domain perturbation technique.

\begin{description}
\item[Notation] 
$\| \cdot \|_D$ denotes the usual norm of $H^1(D)$ for a bounded domain $D$ with a smooth boundary.  $\| \cdot \|_{\Omega}$ is written by $\| \cdot \|$. $H^1_0(\Omega)=\{ u\in H^1(\Omega) : u=0 \mbox{ on } \partial\Omega\}$. 
$\| \cdot\|_{C(\overline{\Omega})}$ denotes the usual norm of  $C(\overline{\Omega})$. $u_n \rightharpoonup u_\infty$ means that $u_n$ weakly converges to $u_\infty$ in $H^1(\Omega)$ or $H^1_0(\Omega)$. $\int_{\Omega} f dx$ for $f \in L^1(\Omega)$ and $\int_{\partial\Omega}g d\sigma$ for $g\in L^1(\partial\Omega)$ are written by $\int_{\Omega}f$ and $\int_{\partial\Omega}g$, respectively. 
$|\cdot|$ represents both the Lebesgue measure in $\Omega$ and the surface measure on $\partial\Omega$. 
\end{description}

\section{Energy approach} 

\label{sec:ener}

First, we prove the existence of an upper bound of $\| \cdot\|_{C(\overline{\Omega})}$ for the positive solutions of the problem 
\begin{align} \label{p-ab}
\begin{cases}
-\Delta u = \beta u-u^{p} & \mbox{ in } \Omega, \\
\frac{\partial u}{\partial \nu}  = -\lambda (u+\alpha)^{q-1}u & \mbox{ on } \partial\Omega, 
\end{cases} 
\end{align}
where $\alpha\geq0$ and $0<\beta\leq1$. 
When $\alpha = 0$ and $\beta = 1$, \eqref{p-ab} is considered as \eqref{p}. 
Proposition \ref{lem:albet:bddnorm} holds unconditionally in the values of  $\beta_\Omega$, $\alpha$, or $\beta$.

\begin{prop} \label{lem:albet:bddnorm} 
If $u$ is a positive solution of \eqref{p-ab} for $\lambda >0$, then $u<1$ in $\overline{\Omega}$. 
\end{prop}

\begin{proof}
Let $M\geq1$ be a constant, and then, $-\Delta M =0\geq f_{\beta}(M)$ in $\Omega$ where 
$f_\beta (t)=\beta t-t^{p}$ for $t\geq0$. 
For $\lambda>0$, we assume by contradiction that $M:= \max_{\overline{\Omega}}u\geq1$ for a positive solution $(\lambda, u)$ of \eqref{p-ab}. 
Assume that $u(x_0)=M$ for some $x_0\in \Omega$, and then, $K>0$ should be selected such that $Kt + f_\beta (t)$ increases for $t\in[0, M]$. Hence, we deduce that 
\begin{align*}
(-\Delta + K)(M-u)\geq (KM+f_\beta (M))-(Ku+f_\beta (u))\geq 0 \ \mbox{ in } 
\Omega. 
\end{align*}
Since $u \in C^2(\Omega)\cap C(\overline{\Omega})$, the strong maximum principle applies, and then, $M-u$ is identically equal to zero in $\overline{\Omega}$, i.e., $u\equiv M$ in $\overline{\Omega}$, which is contradictory for the nonlinear boundary condition. Hence, $x_0 \in \partial\Omega$, i.e., $u(x_0)=M\geq1$. For this, 
we know that $u \in C^1$ in a neighborhood of $x_0$ by the bootstrap argument based on the fact that $u \in W^{1,r}(\Omega)$ with $r>N$; thus, Hopf's boundary point lemma applies at $x_0$. Then, we arrive at the contradiction that 
\begin{align*}
0>-\lambda (u(x_0)+\alpha)^{q-1}u(x_0) =\frac{\partial u}{\partial \nu}(x_0)>0. 
\end{align*} 
\end{proof}

For $0<\beta\leq1$, let 
\[
E_\beta (u)=\int_\Omega \biggl( |\nabla u|^2 - \beta u^2 \biggr), \quad u \in H^1(\Omega), 
\]
where $E_{\beta}(\cdot)$ is written by $E(\cdot)$ when $\beta = 1$. 
Note that $E(u)\leq E_{\beta}(u)$. 
Then, the following lemma is useful in the sequel.

\begin{lem} \label{lem1719}
Let $\{ u_n \}\subset H^1(\Omega)$ satisfy that $E(u_n)\leq0$, $u_n\rightharpoonup u_\infty$, and $u_n \rightarrow u_\infty$ in $L^2(\Omega)$. 
Then, $u_\infty\neq 0$ if $\| u_n\|\geq C$ for some $C>0$. 
\end{lem}

\begin{proof} 
By the weak lower semi--continuity, $E(u_\infty)\leq \varliminf_{n} E(u_n)\leq \varlimsup_{n} E(u_{n}) \leq0$. If $u_{\infty}=0$, then $\| u_{n} \| \rightarrow 0$, as desired. 
\end{proof}

In the sequel, we consider positive solutions of \eqref{p-ab} with $\alpha = 0$: 
\begin{align} \label{p-al0}
\begin{cases}
-\Delta u = \beta u-u^{p} & \mbox{ in } \Omega, \\
\frac{\partial u}{\partial \nu}  = -\lambda u^{q} & \mbox{ on } \partial\Omega. 
\end{cases}    
\end{align}
Then, we start by proving the following two propositions, which provide the asymptotic profile of a positive solution of \eqref{p-al0} as $\lambda\to \infty$.

\begin{prop} \label{prop:asympt}
Assume that $\beta_\Omega =1$. Let $(\lambda_n, \beta_n, u_n)$ be a positive solution of \eqref{p-al0} with $\lambda_n \to\infty$ and $\beta_0\leq \beta_n \leq1$ for some $0<\beta_{0}<1$. Then, $u_n \rightarrow 0$ in $H^1(\Omega)$. 
\end{prop}

\begin{proof} 
By definition, 
\begin{align} \label{121404}
\int_{\Omega} \biggl( \nabla u_n \nabla \varphi - \beta_n u_n\varphi + u_n^p \varphi \biggr) + \lambda_n \int_{\partial\Omega} u_n^{q} \varphi=0, \quad \varphi \in H^1(\Omega). 
\end{align}
Substitute $\varphi = u_n$ into \eqref{121404}, and then, 
\begin{align} \label{121004}
\int_{\Omega} \biggl( |\nabla u_n|^2 - \beta_n u_n^2 + u_n^{p+1} \biggr) + \lambda_n \int_{\partial\Omega} u_n^{q+1}=0. 
\end{align}
Because $u_{n}<1$ in $\overline{\Omega}$ from Proposition \ref{lem:albet:bddnorm}, \eqref{121004} provides 
\begin{align*} 
\int_{\Omega}|\nabla u_n|^2 \leq \beta_n \int_{\Omega} u_n^2 \leq |\Omega|;  
\end{align*}
thus, $\| u_n\|$ is bounded. Up to a subsequence, $u_n \rightharpoonup u_\infty\geq0$, $u_n \rightarrow u_\infty$ in $L^2(\Omega)$ and $L^2(\partial\Omega)$, and $u_n \rightarrow u_\infty$ 
a.e.\ in $\Omega$ for some $u_\infty \in H^1(\Omega)$. 
We may assume that $\beta_n \rightarrow \beta_\infty \in [\beta_0, 1]$. 
\eqref{121004} also provides 
\begin{align*}       
\int_{\partial\Omega}u_n^{q+1} \leq  
\frac{\beta_n}{\lambda_n}\int_{\Omega} u_n^2 \longrightarrow 0, 
\end{align*}
which implies that $\int_{\partial\Omega}u_\infty^{q+1}=0$, and $u_\infty\in H^1_0(\Omega)$. 

Plugging $\varphi \in H^{1}_{0}(\Omega)$ into \eqref{121404}, 
\begin{align*}
\int_{\Omega} \biggl( \nabla u_n \nabla \varphi - \beta_n u_n \varphi + u_n^p \varphi \biggr) = 0. 
\end{align*}
Taking the limit, $u_\infty$ is a nonnegative solution of the Dirichlet problem
\begin{align*}
\begin{cases}
-\Delta u = \beta_\infty u-u^{p} & \mbox{ in } \Omega, \\
u=0 & \mbox{ on } \partial\Omega, 
\end{cases}    
\end{align*}
where we used the Lebesgue dominated convergence theorem to deduce that $\int_{\Omega} u_n^p \varphi \rightarrow \int_{\Omega} u_\infty^p \varphi$. Since $\beta_\infty \leq 1= 
\beta_\Omega$, $u_\infty=0$ (\cite{BO86}).  
Lastly, Lemma \ref{lem1719} applies, and $u_n \rightarrow 0$ in $H^1(\Omega)$ because 
$E(u_n)\leq E_{\beta_n}(u_n)\leq0$. 
\end{proof}

When $\beta_\Omega=1$, we observe from \eqref{lamOcha} that $E(u)\geq 0$ for $u\in H^{1}_{0}(\Omega)$. Moreover, we find that 
$\left\{ u\in H^{1}_{0}(\Omega) : E(u)=0 \right\} = \langle \phi_\Omega \rangle := 
\{ s \phi_\Omega : s \in \mathbb{R} \}$. Then, we investigate the asymptotic profile of a positive solution $u$ of \eqref{p-al0} satisfying $\| u \|\rightarrow 0$.

\begin{prop} \label{prop:unprofi}
Assume that $\beta_\Omega=1$. Let $(\lambda_n, \beta_n, u_n)$ be a positive solution of \eqref{p-al0} such that $\lambda_n \geq \underline{\lambda}$ for some $\underline{\lambda}>0$, $\beta_{0}\leq \beta_{n} \leq 1$ for some $0<\beta_{0}<1$, and $\| u_n \| \rightarrow 0$.  Then, $\frac{u_n}{\| u_n \|} \rightarrow \phi_\Omega$ in $H^1(\Omega)$. 
\end{prop}

\begin{proof}
For $w_n = \frac{u_n}{\| u_n\|}$, $\| w_n \|=1$; then, up to a subsequence, $w_n \rightharpoonup w_\infty\geq0$ and $w_n \rightarrow w_\infty$ in $L^2(\Omega)$ and $L^2(\partial\Omega)$ for some $w_\infty \in H^1(\Omega)$. Then, 
\eqref{121004} deduces 
\begin{align*} 
\underline{\lambda}\int_{\partial\Omega} u_n^{q+1} \leq 
\int_\Omega u_n^2, 
\end{align*}
and then, the condition $\| u_n\| \rightarrow 0$ deduces
\begin{align*}
\int_{\partial\Omega}w_n^{q+1}\leq \frac{\| u_n\|^{1-q}}{\underline{\lambda}}\int_\Omega w_n^2 \longrightarrow 0. 
\end{align*}
Thus, $\int_{\partial\Omega}w_\infty^{q+1}=0$, and $w_\infty \in H^1_0(\Omega)$. 

Since $E(w_n)\leq 0$, we exploit the weak lower semi--continuity and the condition $\beta_\Omega=1$ to deduce that 
\begin{align*}
0\leq E(w_\infty)\leq \varliminf_{n\to \infty} E(w_n) \leq \varlimsup_{n\to \infty} E(w_n)\leq0, 
\end{align*}
which implies that $E(w_n)\rightarrow E(w_\infty)=0$, i.e., 
$\| w_n\| \rightarrow \| w_\infty\|$. 
Since $w_{n} \rightharpoonup w_{\infty}$, $w_n \rightarrow w_\infty$ in $H^1(\Omega)$; therefore, $w_\infty = \phi_\Omega$.  The desired conclusion follows because $\phi_\Omega$ is unique.  
\end{proof}

For further analysis of a positive solution $(\lambda_n, \beta_n, u_n)$ of \eqref{p-al0} with the conditions of Proposition \ref{prop:unprofi}, we use the orthogonal decomposition $H^1(\Omega) = \langle \phi_\Omega \rangle \oplus V$ by $\langle \phi_\Omega \rangle$, where $V$ denotes the orthogonal complement of $\langle \phi_\Omega \rangle$ given explicitly as 
\begin{align*}
V = \left\{ v \in H^1(\Omega) : \int_\Omega \biggl(\nabla v \nabla \phi_\Omega + v \phi_\Omega \biggr) = 0 \right\}. 
\end{align*}
We note that $\langle \phi_\Omega \rangle$ and $V$ are both closed subspaces of $H^1(\Omega)$, and $\|u\|$ is equivalent to $|s|+\|v\|$ for $u=s\phi_\Omega + v \in H^1(\Omega)=\langle \phi_\Omega\rangle\oplus V$. 
When assuming $\beta_\Omega = 1$, we consider the orthogonal decomposition 
\begin{align} \label{vn1557}
u_n = s_n \phi_\Omega + v_n \in \langle \phi_\Omega \rangle \oplus V \vspace{3pt} \end{align}  
for a positive solution $(\lambda_n, \beta_{n}, u_n)$ of \eqref{p-al0} meeting the conditions of Proposition \ref{prop:unprofi}. 
Since $\frac{u_n}{\| u_n\|} \rightarrow \phi_\Omega$ in $H^1(\Omega)$ from Proposition \ref{prop:unprofi}, it follows that 
\begin{align}
&\frac{s_n}{\|u_n\|}\longrightarrow 1, \label{sn1145}\\
&\frac{\| v_n\|}{\| u_n\|}\longrightarrow 0, \label{vn1145}\\
&\frac{\| v_n\|}{s_n} \longrightarrow 0. \label{vnsn1145}
\end{align}
Because of \eqref{sn1145}, we may assume that $s_n>0$. Note that $v_n\geq 0$ on $\partial\Omega$ because $\phi_\Omega=0$ on $\partial\Omega$. 

Then, we prove the following {\it crucial} result. 

\begin{lem}  \label{prop:E} 
Assume that $\beta_\Omega=1$. Let $\{ v_n\}\subset V$ be as introduced by \eqref{vn1557}. Then, there exists $c>0$ such that 
\begin{align} \label{Evn1911}
E(v_n) + c \int_{\partial\Omega} v_n^{q+1}\leq 0 \quad\mbox{ 
for a sufficiently large $n$, }
\end{align}
provided that one of the following conditions is satisfied. 
\begin{enumerate} \setlength{\itemsep}{3pt} 
    \item[(a)] $pq<1$,
    \item[(b)] $pq=1$ and $\lambda_n \to \infty$,
    \item[(c)] $\beta_{n}=1$, $pq>1$, and $\lambda_n$ is bounded from above. 
\end{enumerate}
\end{lem}

\begin{proof}
Because $E_{\beta_{n}}(u_{n})\geq E(u_{n})$, \eqref{121004} deduces 
\[
E(u_n) + \int_\Omega u_{n}^{p+1} +  \lambda_{n}\int_{\partial\Omega}u_{n}^{q+1}\leq0. 
\]
Substituting $u_n=s_n\phi_\Omega + v_n$ yields 
\begin{align} \label{eq1211}
2s_n \left(\int_\Omega \nabla \phi_\Omega \nabla v_n - \phi_\Omega v_n \right) + E(v_n) + \int_\Omega (s_n\phi_\Omega + v_n)^{p+1} + \lambda_n \int_{\partial\Omega} v_n^{q+1} \leq 0. 
\end{align}
By the use of the divergence theorem for $\int_{\Omega}-\Delta \phi_{\Omega}v_{n}$, 
\begin{align} \label{vn1524}
\int_\Omega \phi_\Omega v_n = \int_\Omega -\Delta \phi_\Omega v_n = \int_\Omega \nabla \phi_\Omega \nabla v_n + \int_{\partial\Omega} \biggl( -\frac{\partial \phi_\Omega}{\partial \nu} \biggr)  v_n. 
\end{align}
Thus, \eqref{eq1211} implies that
\begin{align*}
-2s_n \int_{\partial\Omega} \biggl( -\frac{\partial \phi_\Omega}{\partial \nu} \biggr) 
v_n + E(v_n) + \int_\Omega (s_n\phi_\Omega + v_n)^{p+1} + \lambda_n \int_{\partial\Omega} v_n^{q+1} \leq 0,      
\end{align*}
from which 
\begin{align} \label{En1034}
E(v_n) + \frac{\lambda_n}{2} \int_{\partial\Omega} v_n^{q+1} + I_n \leq 0 
\end{align}
with 
\begin{align} \label{In1034}
I_n = \frac{\lambda_n}{2}\int_{\partial\Omega} v_n^{q+1} - 2s_n \int_{\partial\Omega} \left( -\frac{\partial \phi_\Omega}{\partial \nu} \right)v_n.  
\end{align}
Once we get 
\begin{align} \label{In1338}
I_n \geq 0 \quad \mbox{ for a sufficiently large $n$, }
\end{align}
we obtain \eqref{Evn1911} since $\lambda_{n}\geq\underline{\lambda}$, and the proof is complete. 

To verify \eqref{In1338}, we plug $\varphi = \phi_\Omega$ into \eqref{121404}: 
\begin{align} \label{unphOm1502}
\int_\Omega \biggl(\nabla u_{n} \nabla \phi_\Omega - \beta_{n}u_{n} \phi_\Omega + u_n^{p} \phi_\Omega \biggr) = 0.    
\end{align}
Substitute $u_n = s_n \phi_\Omega + v_n$ into \eqref{unphOm1502}, and we deduce
\begin{align*}
0&=\int_\Omega \biggl(\nabla (s_{n}\phi_{\Omega}+v_{n}) \nabla \phi_\Omega 
- (s_{n}\phi_{\Omega}+v_{n}) \phi_\Omega + (1- \beta_{n})(s_{n}\phi_{\Omega}+v_{n})  \phi_\Omega + (s_{n}\phi_{\Omega}+v_{n})^{p} \phi_\Omega \biggr) \\
&=\int_{\Omega}\biggl( \nabla v_{n}\nabla \phi_{\Omega} - v_{n}\phi_{\Omega}  + (1- \beta_{n})(s_{n}\phi_{\Omega}+v_{n})  \phi_\Omega + (s_{n}\phi_{\Omega}+v_{n})^{p} \phi_{\Omega} 
\biggr). 
\end{align*}
Then, \eqref{vn1524} deduces 
\begin{align} 
\int_{\partial\Omega} \left(-\frac{\partial \phi_\Omega}{\partial \nu} \right) \frac{v_n}{s_n^p} 
&= \frac{1-\beta_{n}}{s_{n}^{p-1}} \int_{\Omega}\left( \phi_{\Omega} + \frac{v_{n}}{s_{n}} \right) 
+ \int_\Omega \left( \phi_\Omega + \frac{v_n}{s_n} \right)^{p} \phi_\Omega
\nonumber \\
&\geq \int_\Omega \left( 
\phi_{\Omega} + \frac{v_{n}}{s_{n}} \right)^{p} \phi_{\Omega}. \label{vnsn1033}
\end{align}
For \eqref{vnsn1033}, the equality holds if $\beta_{n}=1$ for all $n$. 

\vspace{6pt}

The proof proceeds by dividing it into two cases. 

\vspace{3pt}

(i) For either case (a) or (b), we use \eqref{vnsn1145} to deduce by the use of Brezis--Lieb's lemma \cite[p.\ 123]{Br11} that
\begin{align} \label{BL}
\int_{\Omega} \left( 
\phi_\Omega + \frac{v_n}{s_n} \right)^{p+1} \longrightarrow \int_{\Omega} \phi_{\Omega}^{p+1}, 
\end{align}
from which we infer that 
\begin{align*}
\int_\Omega \left( 
\phi_\Omega + \frac{v_n}{s_n} \right)^{p} \phi_\Omega 
\longrightarrow \int_\Omega \phi_\Omega^{p+1} > 0. 
\end{align*}
Hence, considering \eqref{bdry1513}, we may derive from \eqref{vnsn1033} that 
\begin{align*}
c s_n^p \leq \int_{\partial\Omega} v_n
\end{align*}
for some $c>0$. 
Here and in what follows, $c, \tilde{c}$, and so on represent generic positive constants that may change step by step. By H\"older's inequality, 
\begin{align} \label{c41554}
c s_n^{p} \leq |\partial \Omega|^{\frac{q}{q+1}}\left( \int_{\partial\Omega}v_n^{q+1} \right)^{\frac{1}{q+1}}. 
\end{align}
Using H\"older's inequality for \eqref{In1034}, 
\begin{align*}
I_n \geq \frac{\lambda_n}{2} \int_{\partial\Omega}v_n^{q+1} - c\, s_n \left(\int_{\partial\Omega} v_n^{q+1}\right)^{\frac{1}{q+1}} 
=\left\{ \frac{\lambda_n}{2}\left( \int_{\partial\Omega} v_n^{q+1} \right)^{\frac{q}{q+1}} - c\, s_n \right\}
\left( \int_{\partial\Omega} v_n^{q+1} \right)^{\frac{1}{q+1}}.  
\end{align*}
By considering \eqref{c41554}, it follows that
\begin{align*}
I_n \geq \biggl( \tilde{c}\, \lambda_n \, s_n^{pq} - c\, s_n \biggr) \left( \int_{\partial\Omega} v_n^{q+1} \right)^{\frac{1}{q+1}} 
= s_n^{pq}\biggl( \tilde{c} \lambda_n - c\, s_n^{1-pq} \biggr) \left( \int_{\partial\Omega} v_n^{q+1} \right)^{\frac{1}{q+1}}. 
\end{align*}
Since $s_n \to 0$ from \eqref{sn1145}, assertion \eqref{In1338} follows. 

\vspace{3pt}

(ii) For case (c), in terms of $\beta_{n}=1$, the combination of \eqref{In1034} with \eqref{vnsn1033} provides  
\begin{align} \label{In1057}
I_n = s_n^{p+1} \biggl\{ \frac{\lambda_n}{2}\int_{\partial\Omega} \frac{v_n^{q+1}}{s_n^{p+1}} - 2 \int_\Omega \biggl( \phi_\Omega + \frac{v_n}{s_n}\biggr)^p \phi_\Omega \biggr\}. 
\end{align}
Plugging $\varphi = 1$ into \eqref{121404} with $\beta_{n}=1$, 
\begin{align*}
-\int_\Omega u_n + \int_\Omega u_n^p + \lambda_n \int_{\partial\Omega}u_n^q = 0.      
\end{align*}
Substituting $u_n = s_n\phi_\Omega + v_n$, 
\begin{align*} 
-\int_\Omega \biggl(\phi_\Omega + \frac{v_n}{s_n} \biggr) + s_n^{p-1}\int_\Omega \biggl( \phi_\Omega + \frac{v_n}{s_n}\biggr)^p + \frac{\lambda_{n}}{s_{n}}\int_{\partial\Omega} v_{n}^{q}=0.  
\end{align*}
Taking \eqref{vnsn1145} and \eqref{BL} into account, this implies 
\begin{align*}
\frac{\lambda_{n}}{s_{n}}\int_{\partial\Omega} v_{n}^{q}\longrightarrow \int_{\Omega} \phi_{\Omega} >0. 
\end{align*}
Thus, we may deduce 
\begin{align*}
c\frac{s_n}{\lambda_n} \leq \int_{\partial\Omega} v_n^{q}.  
\end{align*}
Using H\"older's inequality, 
\begin{align} \label{vnq+1058}
c \left( \frac{s_n}{\lambda_n} \right)^{\frac{q+1}{q}} \leq \int_{\partial\Omega}v_n^{q+1}. 
\end{align}
Combining \eqref{In1057} with \eqref{vnq+1058}, 
\begin{align*}
I_n\geq s_n^{p+1} \biggl\{ c\,  s_n^{\frac{1}{q}-p}\lambda_n^{-\frac{1}{q}}
-2\int_\Omega \biggl( \phi_\Omega + \frac{v_n}{s_n}\biggr)^p \phi_\Omega \biggr\}. 
\end{align*}
Since $pq>1$, $s_n \to 0$, and $\lambda_n$ is bounded from above, we observe that 
$s_n^{\, \frac{1}{q}-p} \rightarrow \infty$, $\lambda_n^{-\frac{1}{q}} \geq \tilde{c}$, and 
$\int_\Omega \left( \phi_\Omega + \frac{v_n}{s_n}\right)^p \phi_\Omega \rightarrow \int_\Omega \phi_\Omega^{p+1} > 0$. 
Thus, assertion \eqref{In1338} follows. 
\end{proof}

\section{A priori upper bounds of $\lambda$ for positive solutions} 

\label{sec:bounds}

In this section, we establish {\it a priori} upper bounds of $\lambda$ 
for a positive solution $(\lambda, u)$ of \eqref{p-ab} with $\lambda>0$.

\begin{prop} \label{153319}
Assume that $\beta_\Omega=1$. Then, the following {\rm two} assertions hold. 
\begin{enumerate} \setlength{\itemsep}{3pt} 

\item Given $0<\beta<1$, there exists $\Lambda_{\alpha, \beta}>0$ such that $\lambda \leq \Lambda_{\alpha, \beta}$ if problem \eqref{p-ab} has a positive solution for $\lambda > 0$. Moreover, $\Lambda_{\alpha, \beta}$ is determined uniformly in $\alpha \in [0, \alpha_{0}]$ for $\alpha_{0}>0$ (say $\Lambda_{\beta}$). 

\item Let $\alpha=0$. Assume additionally that $pq\leq1$. Then, the same conclusion remains valid. Moreover, $\Lambda_{0,\beta}$ is determined uniformly in $\beta \in [\beta_0, 1]$ for $0<\beta_0<1$ (say $\Lambda_{0}$). 
\end{enumerate}
\end{prop}

\begin{proof}
(i) Without loss of generality, we assume $\alpha_{0}=1$. 
We argue by contradiction. Assume that problem \eqref{p-ab} has a positive solution $(\lambda_n, \alpha_n, u_n)$ with 
$\lambda_n \to \infty$ and $\alpha_{n} \in [0,1]$, and then, 
\begin{align} \label{1638}
    \int_\Omega \biggl( |\nabla u_n|^2 - \beta u_n^2 + u_n^{p+1} \biggr) + \lambda_n \int_{\partial\Omega} (u_n + \alpha_n)^{q-1}u_n^2 = 0. 
\end{align}
From \eqref{1638}, 
Proposition \ref{lem:albet:bddnorm} ensures that $u_n$ is bounded in $H^1(\Omega)$ because 
\begin{align*}
    \int_\Omega |\nabla u_n|^2 \leq \beta \int_\Omega u_n^2 \leq \beta  |\Omega|. 
\end{align*}
Hence, up to a subsequence, $u_n \rightharpoonup u_\infty\geq0$, and $u_n \to u_\infty$ in $L^{2}(\Omega)$ and $L^2(\partial\Omega)$ 
for some $u_\infty \in H^1(\Omega)$. 
Since $u_n \leq 1$ in $\overline{\Omega}$ and $0\leq \alpha_n \leq 1$, \eqref{1638} provides 
\begin{align*} 
\lambda_n 2^{q-1} \int_{\partial\Omega} u_n^2 \leq \beta |\Omega|. 
\end{align*}
Passing to the limit, $\int_{\partial\Omega} u_n^2 \rightarrow 0$, which implies that $\int_{\partial\Omega}u_\infty^2 = 0$, and $u_\infty \in H^1_0(\Omega)$. 

Then, from \eqref{1638}, $E_\beta(u_n)\leq 0$. By the weak lower semi--continuity, 
\begin{align*}
    E_\beta(u_\infty) \leq \varliminf_n E_\beta (u_n) \leq \varlimsup_n E_\beta (u_n) \leq0. 
\end{align*}
We see that $0\leq E(u_\infty)\leq E_\beta (u_\infty)$ because $\beta_\Omega=1$ and $u_\infty \in H^{1}_{0}(\Omega)$; thus, $E_{\beta}(u_{\infty})=E(u_{\infty})=0$, i.e., 
$u_{\infty}=0$. Moreover, we observe $E_\beta (u_n) \rightarrow 0$; thus, $\| u_n \| \rightarrow 0$. 
Then, a similar argument is repeated for $w_n = \frac{u_n}{\| u_n\|}$. 
Note that $\| w_n\|=1$ and $E_\beta (w_n)\leq 0$. Up to a subsequence, $w_n \rightharpoonup w_\infty\geq0$, and $w_n \to w_\infty$ in $L^2(\Omega)$ and $L^2(\partial\Omega)$ for some $w_\infty \in H^1(\Omega)$. 
Then, \eqref{1638} provides 
\begin{align*}
    \lambda_n 2^{q-1} \int_{\partial\Omega} w_n^2 \leq \beta \int_\Omega w_n^2 \leq \beta. 
\end{align*}
Taking the limit, $\int_{\partial\Omega}w_n^2 \rightarrow 0$; thus, $\int_{\partial\Omega} w_\infty^2=0$, and $w_\infty \in H^1_0(\Omega)$. By using  the weak lower semi--continuity and the condition $\beta_\Omega=1$, 
\begin{align*}
    0\leq E(w_\infty)\leq E_\beta(w_\infty)\leq \varliminf_n E_\beta (w_n)\leq \varlimsup_n E_\beta (w_n) \leq 0; 
\end{align*}
thus, $\| w_n \| \rightarrow 0$, which is contradictory for the assertion $\| w_n \|=1$. Assertion (i) has been verified. 

\vspace{3pt}

(ii) By contradiction we may assume that there exists a positive solution $(\lambda_n, \beta_n, u_n)$ of \eqref{p-ab} with $\alpha=0$ such that $\lambda_n \to \infty$ and $\beta_{n}\in [\beta_{0},1]$. Then, Proposition \ref{prop:asympt} shows that $\| u_n\|\rightarrow 0$, 
and moreover, Proposition \ref{prop:unprofi} shows that $w_n = \frac{u_n}{\| u_n\|}\rightarrow \phi_\Omega$ in $H^1(\Omega)$. We employ the orthogonal decomposition $u_n = s_n\phi_\Omega + v_n \in \langle \phi_\Omega \rangle + V$ for $u_n$. 
Since $\| u_n\| \rightarrow 0$, so does $\| v_n\|$ from \eqref{vn1145}. For $\psi_n = \frac{v_n}{\| v_n\|}\in V$, $\| \psi_n\|=1$; then, up to a subsequence, $\psi_n \rightharpoonup \psi_\infty$ and $\psi_n \rightarrow \psi_\infty$ in $L^2(\Omega)$ and $L^2(\partial\Omega)$ for some $\psi_\infty \in H^1(\Omega)$. Note that $\psi_\infty\geq0$ on $\partial\Omega$ because $v_n\geq0$ on $\partial\Omega$. 

Then, by applying Lemma \ref{prop:E}(a), (b), 
\eqref{Evn1911} deduces 
\begin{align*}
    c\int_{\partial\Omega} \psi_n^{q+1} \leq -E(\psi_n) \| v_n\|^{1-q} \longrightarrow 0,   
\end{align*}
which implies that $\int_{\partial\Omega} \psi_n^{q+1} \rightarrow 0$;  thus, 
$\int_{\partial\Omega} \psi_\infty^{q+1}=0$, and $\psi_\infty \in H^1_0(\Omega)$. 
Since $E(\psi_n)\leq 0$ from \eqref{Evn1911}, the similar argument as 
in the second paragraph of the proof of Proposition \ref{prop:unprofi} yields that $\psi_n \rightarrow \phi_{\Omega}$ in $H^1(\Omega)$. We deduce as well that $\phi_{\Omega} \in V$ since $V$ is closed. This is contradictory to $\langle \phi_{\Omega} \rangle \cap V=\{ 0\}$. 
\end{proof}

\section{Limiting behavior of positive solutions} 

\label{sec:asymp}

In this section, we investigate the limiting behaviors of a positive solution $(\lambda, u)$ of \eqref{p-ab} as $\lambda \to 0^{+}$ and as $\| u\|_{C(\overline{\Omega})} \to 0$, respectively. 

\subsection{Limiting behavior as $\lambda \to 0^{+}$} 
Let $\alpha \geq 0$ and $0<\beta \leq 1$ be fixed, and let $(\lambda_n, u_n)$ be a positive solution of \eqref{p-ab} such that $\lambda_n \to 0^{+}$. 
On the basis of Proposition \ref{lem:albet:bddnorm}, we employ the bootstrap argument attributed to elliptic regularity \cite[Theorem 2.2]{Ro2005} and a compactness argument to deduce that, up to a subsequence, $u_{n} \rightarrow u_{\infty}$ in $C(\overline{\Omega})$ for some $u_{\infty}\geq0$. 

Then, we prove the following. 

\begin{prop} \label{215925}
If $(\lambda_n, u_n)$ is a positive solution of \eqref{p-ab} such that $\lambda_n \to 0^{+}$, then, up to a subsequence, either $u_n \rightarrow \beta^{\frac{1}{p-1}}$ or $u_n \rightarrow 0$ in $C(\overline{\Omega})$. 
\end{prop}

\begin{proof}
We assume that $\|u_n \|_{C(\overline{\Omega})}\geq \delta$ for some $\delta>0$, and then, $u_{\infty}\neq 0$. By definition, 
\begin{align*}
    \int_\Omega \biggl( \nabla u_n \nabla \varphi - \beta u_n \varphi + u_n^p \varphi \biggr) + \lambda_n \int_{\partial \Omega} (u_n + \alpha)^{q-1}u_n \varphi = 0, \quad \varphi \in H^1(\Omega). 
\end{align*}
Substitute $\varphi = u_{n}$, and 
Proposition \ref{lem:albet:bddnorm} ensures
\begin{align*}
    \int_\Omega |\nabla u_n|^2 \leq \beta \int_\Omega u_n^2 \leq |\Omega|,  
\end{align*}
i.e., $u_n$ is bounded in $H^1(\Omega)$; therefore, up to a subsequence, 
$u_n \rightharpoonup u\in H^{1}(\Omega)$, $u_n \rightarrow u$ in $L^2(\Omega)$ and $L^2(\partial\Omega)$, and $u_n \rightarrow u$ a.e.\ in $\Omega$ for some $u\in H^{1}(\Omega)$; thus, $u=u_{\infty}\in H^{1}(\Omega)$. 
Since $u_{n}\leq 1$, we obtain that for $\alpha >0$, 
\begin{align*}
(u_n + \alpha)^{q-1}u_n = \biggl( \frac{u_n}{u_n + \alpha} \biggr)^{1-q} u_n^q \leq 1, 
\end{align*}
so that $\lambda_n \int_{\partial \Omega} (u_n + \alpha)^{q-1}u_n \varphi \rightarrow 0$. Then, passing to the limit, 
\begin{align} \label{220925}
    \int_{\Omega} \biggl( \nabla u_{\infty} \nabla \varphi - \beta u_{\infty} \varphi + u_{\infty}^{p} \varphi \biggr)=0, 
\end{align}
where we used the Lebesgue dominated convergence theorem to infer that $\int_\Omega u_n^p \varphi \rightarrow \int_{\Omega} u_{\infty}^{p} \varphi$. We find from \eqref{220925} that 
$u_{\infty}\neq 0$ is a nonnegative weak solution of the Neumann  problem 
\begin{align*}  
\begin{cases}
-\Delta u = \beta u- u^p & \mbox{ in } \Omega, \\
\frac{\partial u}{\partial \nu}  = 0 & \mbox{ on } \partial\Omega, 
\end{cases} 
\end{align*}
and thus, $u_{\infty}=\beta^{\frac{1}{p-1}}$ as desired. 
\end{proof}

\subsection{Limiting behavior as $\| u \|_{C(\overline{\Omega})}\to 0$} 
We consider the limiting behavior of a positive solution $(\lambda, u)$ of \eqref{p-ab} with $\| u\|_{C(\overline{\Omega})} \rightarrow 0$ in the case  
when $\alpha = 0$ and $0<\beta\leq1$, that is, we consider a positive solution $(\lambda, u)$ of \eqref{p-al0} with $\| u \|_{C(\overline{\Omega})} \rightarrow 0$. We argue the case $\alpha > 0$ in the next section using the general theory for local bifurcation from zero.

For $0<\beta < 1$, we have the following. 
\begin{prop} \label{222825} 
Assume that $\beta_\Omega = 1$. Let $0<\beta<1$. If $(\lambda_n,u_n)$ is a positive solution of \eqref{p-al0} such that $\lambda_n> 0$ is bounded from above and $\| u_n \|_{C(\overline{\Omega})}\rightarrow 0$, then $\lambda_n \to 0$. 
\end{prop}

\begin{proof}
Assume by contradiction that $\lambda_n \rightarrow \overline{\lambda}$ for some $\overline{\lambda}>0$ and 
$\| u_n\|_{C(\overline{\Omega})} \rightarrow 0$ for a positive solution $(\lambda_n,u_n)$ of \eqref{p-al0}. By definition, 
\begin{align} \label{unbeta}
\int_{\Omega} \biggl( |\nabla u_n|^2 - \beta u_n^2 + u_n^{p+1} \biggr) + \lambda_n \int_{\partial\Omega} u_n^{q+1}=0, 
\end{align}
and then, $\| u_{n} \| \rightarrow 0$. For $w_n=\frac{u_n}{\| u_n\|}$, $\| w_n\|=1$; then, up to a subsequence, $w_n \rightharpoonup w_\infty\geq0$, and $w_n \rightarrow w_\infty$ in $L^2(\Omega)$ and $L^2(\partial\Omega)$ for some $w_\infty \in H^1(\Omega)$. Then, \eqref{unbeta} deduces 
\begin{align*}
   \int_{\partial\Omega}w_n^{q+1}\leq \frac{\beta}{\lambda_n}\left( \int_\Omega w_n^2 \right) \| u_n\|^{1-q} \longrightarrow 0, 
\end{align*}
which implies that $\int_{\partial\Omega} w_\infty^{q+1} = 0$, and thus, $w_\infty \in H^1_0(\Omega)$. The rest of the proof proceeds similarly as in the last paragraph of the proof of Proposition \ref{153319}(i), and we arrive at a contradiction. 
\end{proof}

For $\beta=1$, that is, for \eqref{p}, we prove Proposition \ref{prop827} below. To this end, we prove the following {\it three} preparatory lemmas for $U = \lambda^{-\frac{1}{1-q}}u$ with a positive solution $(\lambda, u)$ of \eqref{p} for $\lambda > 0$. 
Note that $U=\lambda^{-\frac{1}{1-q}}u$ is a positive solution of the problem
\begin{align*}
\begin{cases}
-\Delta U = U - \lambda^{\frac{p-1}{1-q}} U^{p} & \mbox{ in } \Omega, \\
\frac{\partial U}{\partial \nu}  = -U^q & \mbox{ on } \partial\Omega. 
\end{cases} 
\end{align*}

\begin{lem} \label{lem1809}
There exists $C>0$ such that $\| U_n\|\leq C$ for $U_n = \lambda_n^{-\frac{1}{1-q}}u_n$ with a positive solution 
$(\lambda_n,u_n)$ of \eqref{p} satisfying that $\lambda_{n}\to 0^{+}$ and $\| u_{n} \| \rightarrow 0$. 
\end{lem}

\begin{proof}
Assume by contradiction that $\| U_n\|\rightarrow \infty$. For $w_n = \frac{U_n}{\| U_n\|}$, $\| w_n\|=1$; up to a subsequence, $w_n \rightharpoonup w_\infty\geq0$, and $w_n \to w_\infty$ in $L^{2}(\Omega)$ and $L^2(\partial\Omega)$ for some $w_\infty \in H^1(\Omega)$.  Since $E(w_n)\leq0$, Lemma \ref{lem1719} provides $w_\infty\neq 0$. 

Recall that $(\lambda_{n},U_{n})$ admits 
\begin{align} \label{Upro1758} 
\int_\Omega \biggl( \nabla U \nabla \varphi - U \varphi + \lambda^{\frac{p-1}{1-q}} U^{p}\varphi \biggr) + \int_{\partial\Omega} U^{q}\varphi = 0, \quad \varphi \in H^1(\Omega). 
\end{align}
Using the test function $\varphi=1$ in \eqref{Upro1758}, 
\begin{align*}
\int_\Omega U_n = \lambda_n^{\frac{p-1}{1-q}}\int_\Omega U_n^p + \int_{\partial\Omega}U_n^q 
= \int_\Omega u_n^{p-1}U_n + \int_{\partial\Omega}U_n^q, 
\end{align*}
which implies
\begin{align} \label{wn1624}
\int_\Omega w_n = \int_\Omega u_n^{p-1}w_n + \int_{\partial\Omega}w_n^q \, \| U_n\|^{q-1}.  \end{align}
We may assume that $u_n \rightarrow 0$ a.e.\ in $\Omega$. 
Since $u_n\leq1$ in $\overline{\Omega}$,  
\[
\int_\Omega u_n^{p-1}w_n = \int_\Omega u_n^{p-1}w_\infty + \int_\Omega u_n^{p-1}(w_n - w_\infty) \longrightarrow 0, 
\]
where we used the Lebesgue dominated convergence theorem to deduce $\int_{\Omega}u_{n}^{p-1}w_{\infty} \rightarrow 0$. Then, taking the limit in \eqref{wn1624} yields $\int_\Omega w_\infty = 0$, and $w_\infty=0$, which is a contradiction.
\end{proof}

\begin{lem} \label{lem1819}
Assume that $\beta_\Omega = 1$. Then, there is no positive solution $U$ of \eqref{Upro1758} for $\lambda=0$.
\end{lem}

\begin{proof}
If it exists, then from \eqref{Upro1758} with $\lambda=0$ and $\varphi=1$, 
we deduce that $0<\int_{\Omega}U = \int_{\partial\Omega}U^{q}$; thus, 
$U>0$ on some $\Gamma \subset \partial\Omega$ with $|\Gamma|>0$, which implies $\int_{\partial\Omega} \frac{\partial \phi_\Omega}{\partial \nu} U<0$. The test function $\varphi = \phi_\Omega$ in \eqref{Upro1758} with $\lambda = 0$ is considered, and then, 
$\int_\Omega (\nabla U \nabla \phi_\Omega - U\phi_\Omega) = 0$. The divergence theorem leads us to the contradiction
\begin{align*}
\int_\Omega \phi_\Omega U = \int_\Omega -\Delta \phi_\Omega U = \int_\Omega \nabla \phi_\Omega \nabla U - \int_{\partial\Omega} \frac{\partial \phi_\Omega}{\partial \nu} U > \int_\Omega \nabla \phi_\Omega \nabla U.
\end{align*}
\end{proof}

\begin{lem} \label{lem1813}
Assume that $\beta_\Omega = 1$ and $pq\geq 1$. Then, 
there exists $C>0$ such that $\| U_n \|\geq C$ for $U_n = \lambda_n^{-\frac{1}{1-q}}u_n$ with a positive solution $(\lambda_n, u_n)$ of \eqref{p} meeting the condition $\lambda_n \to 0^{+}$. 
\end{lem}

\begin{proof}
Assume by contradiction that $\|U_n \|\rightarrow 0$ for a positive solution $(\lambda_{n}, u_{n})$ of \eqref{p} with $\lambda_n \rightarrow 0^{+}$. For $w_n = \frac{U_n}{\| U_n\|}$, $\| w_n\|=1$; up to a subsequence, $w_n \rightharpoonup w_\infty\geq0$, and $w_n \rightarrow w_\infty$ in $L^{p+1}(\Omega)$ and $L^2(\partial\Omega)$ for some $w_\infty \in H^1(\Omega)$. Plugging $(\lambda, U)=(\lambda_n, U_n)$ and $\varphi = U_n$ into \eqref{Upro1758} yields 
\begin{align} \label{Un1819}
\int_\Omega \biggl( |\nabla U_n|^2 - U_n^2 + \lambda_n^{\frac{p-1}{1-q}}U_n^{p+1} \biggr) + \int_{\partial\Omega} U_n^{q+1} = 0.     
\end{align}
Then, we deduce that $\int_{\partial\Omega} w_n^{q+1}\leq \int_{\Omega} w_n^2 \| U_n \|^{1-q} \rightarrow 0$; thus, $\int_{\partial\Omega}w_\infty^{q+1}=0$, and $w_\infty \in H^1_0(\Omega)$. Further, \eqref{Un1819} provides $E(w_n)\leq0$. 
Thus, $w_n \rightarrow \phi_\Omega$ in $H^1(\Omega)$ by arguing similarly as in the second paragraph of the proof of Proposition \ref{prop:unprofi}.

For a contradiction, we exploit the same strategy developed in the proof of Proposition \ref{153319}(ii). For the orthogonal decomposition $U_n = s_n\phi_\Omega + v_n \in \langle \phi_\Omega \rangle \oplus V$ in \eqref{vn1557}, we obtain \eqref{sn1145} to \eqref{vnsn1145} with $u_n$ replaced by $U_n$. Then, the next counterparts of \eqref{En1034} and \eqref{In1034} for \eqref{Un1819} can be deduced by following the line in the proof of Lemma \ref{prop:E}. 
\begin{align}
&E(v_n) + \frac{1}{2} \int_{\partial\Omega} v_n^{q+1} + J_n \leq 0, \quad\mbox{ with } \nonumber \\
&J_n = \frac{1}{2} \int_{\partial\Omega} v_n^{q+1}-2s_n\int_{\partial\Omega} \biggl( 
-\frac{\partial \phi_\Omega}{\partial \nu} \biggr) v_n. \label{Jn1522}
\end{align}
In the same spirit as \eqref{In1338}, we intend to verify that  
\begin{align} \label{Jn}
J_n\geq0 \quad \mbox{ for a sufficiently large $n$. }
\end{align}
Analogously to \eqref{vnsn1033}, we obtain 
\begin{align*}
\int_{\partial\Omega} \biggl( -\frac{\partial \phi_\Omega}{\partial \nu}\biggr) v_n = \lambda_n^{\frac{p-1}{1-q}} s_n^p \int_\Omega \left( 
\phi_\Omega + \frac{v_n}{s_n} \right)^{p} \phi_\Omega, 
\end{align*}
which is used to deduce from \eqref{Jn1522} that 
\begin{align} 
J_n = s^{p+1} \biggl\{ \frac{1}{2} \int_{\partial\Omega} \frac{v_n^{q+1}}{s^{p+1}} - 2\lambda_n^{\frac{p-1}{1-q}}  
\int_\Omega \left(\phi_\Omega + \frac{v_n}{s_n} \right)^{p} \phi_\Omega \biggr\}.  \label{Jn1253}
\end{align} 
Plugging $\varphi = 1$ into \eqref{Upro1758} provides 
\begin{align*}
-\int_\Omega U_n + \lambda_n^{\frac{p-1}{1-q}} \int_\Omega U_n^{p} + \int_{\partial\Omega} U_n^{q} = 0. 
\end{align*}
Substituting $U_n = s_n\phi_\Omega + v_n$, 
\begin{align*}
-\int_\Omega \biggl( \phi_\Omega + \frac{v_n}{s_n} \biggr) + \lambda_n^{\frac{p-1}{1-q}} s_n^{p-1} \int_\Omega \biggl( \phi_\Omega + \frac{v_n}{s_n} \biggr)^p 
+ \int_{\partial\Omega}\frac{v_n^q}{s_n} = 0.
\end{align*}
Taking \eqref{vnsn1145} and \eqref{BL} into account, this implies 
\begin{align*}
\int_{\partial\Omega}\frac{v_n^q}{s_n}  \longrightarrow \int_\Omega \phi_\Omega >0. 
\end{align*}
Hence, we may deduce 
\begin{align*}
c s_n \leq \int_{\partial\Omega} v_n^q.
\end{align*}
By H\"older's inequality, 
\begin{align*}
c s_{n}^{\frac{q+1}{q}}\leq \int_{\partial\Omega}v_{n}^{q+1}.
\end{align*}
Combining this inequality with \eqref{Jn1253} provides 
\begin{align*}
J_n \geq s^{p+1} \biggl\{ c s_n^{\frac{1}{q}-p} - 2\lambda_n^{\frac{p-1}{1-q}}  
\int_\Omega \left(\phi_\Omega + \frac{v_n}{s_n} \right)^{p} \phi_\Omega \biggr\}; 
\end{align*}
thus, \eqref{Jn} follows by the condition $pq\geq1$ since $\lambda_n \rightarrow 0^{+}$ and $s_n \to 0^{+}$. Then, we established 
\begin{align} \label{181004}
    E(v_n) + \frac{1}{2}\int_{\partial\Omega} v_n^{q+1} \leq0 \quad \mbox{ for a sufficiently large $n$.}
\end{align} 
Thanks to \eqref{181004}, the rest of the proof is carried out similarly as the proof of Proposition \ref{153319}(ii). 
\end{proof}

Then, we prove a necessary condition for the existence of bifurcation from $\{ (\lambda,0): \lambda\geq0 \}$ for positive solutions of \eqref{p}.  

\begin{prop} \label{prop827}
Suppose that $\beta_\Omega = 1$. Then, the following {\rm three} assertions are valid. 
\begin{enumerate} \setlength{\itemsep}{3pt} 

\item Assume $pq<1$. Then, it holds that 
\begin{align}
\lim_{n\to \infty} \lambda_n = 0  \label{171325} 
\end{align}
for a positive solution $(\lambda_n,u_n)$ of \eqref{p} such that $\lambda_n> 0$ is bounded from above and $\| u_n\|_{C(\overline{\Omega})}\rightarrow 0$. 
\item Assume $pq=1$. Then, problem \eqref{p} has {\rm no} positive solution $(\lambda,u)$ for $\lambda>0$ in a neighborhood of $(0,0)$ in $[0,\infty)\times C(\overline{\Omega})$, 
implying that there exists $\underline{\lambda}>0$ such that 
\begin{align} \label{171325b}
\underline{\lambda}\leq \lambda_{n} \quad \mbox{ for a sufficiently large $n$}  
\end{align}
if $(\lambda_n, u_n)$ is a positive solution of \eqref{p} such that $\lambda_{n}> 0$ and $\| u_{n} \|_{C(\overline{\Omega})} \rightarrow 0$. 
\item Assume $pq>1$. Then, it holds that 
\begin{align}
\lim_{n\to \infty} \lambda_{n} = \infty  \label{171325c}    
\end{align}
for a positive solution $(\lambda_n,u_n)$ of \eqref{p} such that $\lambda_{n}>0$ and $\| u_n\|_{C(\overline{\Omega})}\rightarrow 0$. 

\end{enumerate}
\end{prop}

\begin{proof} 
(i) For \eqref{171325}, we assume by contradiction that $\lambda_n \rightarrow \lambda_\infty > 0$ and $\| u_n\|_{C(\overline{\Omega})} \rightarrow 0$ for a positive solution $(\lambda_n,u_n)$ of \eqref{p} with $\lambda_n>0$; then, $\| u_n\| \rightarrow 0$. 
This is the case where Proposition \ref{prop:unprofi} applies, and then, Lemma \ref{prop:E}(a) applies; therefore, we have \eqref{Evn1911} for $u_n = s_n \phi_\Omega + v_n \in \langle \phi_\Omega \rangle \oplus V$, following \eqref{vn1557} together with \eqref{sn1145}--\eqref{vnsn1145}. The rest is the same as the proof of Proposition \ref{153319}(ii). 

(ii) For \eqref{171325b}, we employ Lemmas \ref{lem1809} through  \ref{lem1813}. Assume by contradiction that $(\lambda_n, u_n) \rightarrow (0,0)$ in $[0,\infty) \times C(\overline{\Omega})$ for some positive solution $(\lambda_n,u_n)$ of \eqref{p} with $\lambda_n>0$. Then, $\| u_n\| \rightarrow 0$. Problem \eqref{Upro1758} admits the positive solution $(\lambda_n, U_n)$ with 
$U_n = \lambda_n^{-1/(1-q)}u_n$, and then, $U_n$ is bounded in $H^1(\Omega)$ by Lemma \ref{lem1809}. Up to a subsequence, $U_n \rightharpoonup U_\infty\geq0$, and $U_n \rightarrow U_\infty$ in $L^{p+1}(\Omega)$ and $L^2(\partial\Omega)$ for some $U_\infty\in H^1(\Omega)$. Owing to Lemma \ref{lem1813}, Lemma \ref{lem1719} provides that $U_\infty\neq 0$.  
On the other hand, we infer $U_{\infty}=0$. As a matter of fact, 
substitute $(\lambda,U)=(\lambda_n,U_n)$ into \eqref{Upro1758}, and then, taking the limit provides  
\begin{align*}
\int_\Omega \biggl( \nabla U_\infty \nabla \varphi - U_\infty\varphi \biggr) + \int_{\partial\Omega} U_\infty^q \varphi = 0, \quad 
\varphi \in H^1(\Omega). 
\end{align*} 
This implies that $U_\infty$ is a nonnegative weak solution of \eqref{Upro1758} for $\lambda=0$, and thus, Lemma \ref{lem1819} provides the desired assertion, which is a contradiction. 

(iii) For \eqref{171325c}, we assume by contradiction that $(\lambda_n,u_n)$ is a positive solution of \eqref{p} such that $\lambda_n> 0$ is bounded from above and $\| u_n\|_{C(\overline{\Omega})}\rightarrow 0$, and then, the following two possibilities may occur: one is that $\lambda_n \rightarrow \lambda_\infty > 0$, and the other is that $\lambda_n \rightarrow 0$. 
However, the former case does not occur when using the same argument as in item (i), where we used Lemma \ref{prop:E}(c) instead. The latter case does not occur in a similar way when using the same argument as in item (ii). 
\end{proof}

\section{Bounded subcontinua of nonnegative solutions: case $\beta_\Omega=1$ and $pq\leq1$} 

\label{sec:continuum}

In this section, we prove assertions (II) and (III) of Theorem \ref{mt}. 
First, we explore the existence of the subcontinuum $\mathcal{C}_{p,q}^{\ast}$ 
presented in these assertions, and our goal is to verify \eqref{Cpq919}.

Let $\alpha>0$ and $0<\beta <1$. Under the condition $\beta_\Omega = 1$, we study the bifurcation of positive solutions from $\{ (\lambda, 0) : \lambda\geq0\}$ for \eqref{p-abz}. 
Because \eqref{p-abz} is a regular problem for $u\geq -\frac{\alpha}{2}$ in $\overline{\Omega}$, a positive solution $u$ of \eqref{p-abz} belongs to $C^{2+\tau}(\overline{\Omega})$ with $0<\tau<1$, and it is positive in $\overline{\Omega}$.  

When $\beta_\Omega = 1>\beta>0$, we introduce the principal eigenvalue $\lambda_\beta > 0$ of the eigenvalue problem (\cite[Lemma 9]{GM09})
\begin{align} \label{epromu} 
\begin{cases}
-\Delta \varphi = \beta \varphi & \mbox{ in } \Omega, \\
\frac{\partial \varphi}{\partial \nu}  = -\lambda \varphi & \mbox{ on } \partial\Omega. 
\end{cases} 
\end{align}
It is well known that $\lambda_\beta$ is the largest and simple, and its associated eigenfunction $\varphi_\beta$ satisfies $\varphi_\beta >0$ in $\overline{\Omega}$. To look for bifurcation points on $\{ (\lambda, 0)\}$ for positive solutions of \eqref{p-abz}, we consider the linearized eigenvalue problem \eqref{165331} associated to \eqref{p-abz}. 
Substituting $u=0$ into \eqref{165331}, we obtain \eqref{epromu0}, and then, 
in accordance with \eqref{epromu}, problem \eqref{epromu0} has the largest (principal) eigenvalue $\lambda_{\alpha, \beta} > 0$ satisfying $\lambda_\beta = \lambda_{\alpha, \beta} \, \alpha^{q-1}$. Therefore, $\lambda_{\alpha,\beta}$ is simple with its associated eigenfunction $\varphi_{\beta}$, and it holds that 
\begin{align} \label{2026s}
\lambda_{\alpha, \beta} \longrightarrow 0 \quad \mbox{ as } \ \alpha \to 0^{+}.  
\end{align}
For bifurcation analysis, problem \eqref{p-abz} is reduced to an operator equation for a strongly positive mapping in the framework of $C(\overline{\Omega})$ by arguing as in \cite[Section 2] {Um2013}; then, the local bifurcation theory from simple eigenvalues and unilateral global bifurcation theory \cite{CR71, Ra71, LGbook} are applied to deduce that problem \eqref{p-abz} has a {\it component} (i.e., maximal, closed, and connected subset) $\mathcal{C}_{\alpha, \beta}=\{ (\lambda, u) \}$ in $[0, \infty)\times C(\overline{\Omega})$ for the nonnegative solutions bifurcating at $(\lambda_{\alpha,\beta}, 0)$ (cf.\ \cite[Proposition 2.2]{Um2013}). In addition, $u$ is a positive solution of 
\eqref{p-abz} (i.e., \eqref{p-ab}) such that $u>0$ in $\overline{\Omega}$ if $(\lambda,u) \in \mathcal{C}_{\alpha,\beta}\setminus \{ (\lambda_{\alpha, \beta}, 0) \}$, ensured by the strong maximum principle and boundary point lemma \cite{PW67}. More precisely, the positive solution set of \eqref{p-abz} does not meet any point on $\{ (\lambda, 0)\}$ except for $(\lambda_{\alpha, \beta},0)$ because the principal eigenvalue $\lambda_{\alpha, \beta}$ of \eqref{epromu0} is unique. Thus, 
\[
\mathcal{C}_{\alpha, \beta} \cap \{ (\lambda, 0) : \lambda\geq0\} = \{ (\lambda_{\alpha, \beta}, 0)\}. 
\]
Lastly, $\mathcal{C}_{\alpha, \beta}$ is bounded in $[0, \infty) \times C(\overline{\Omega})$ by virtue of Propositions \ref{lem:albet:bddnorm} and \ref{153319}(i), and then, by Proposition \ref{215925}, we deduce that 
\begin{align} \label{221725}
\mathcal{C}_{\alpha, \beta} \cap \{ (0,u) : u\geq0 \} = \{ (0, \beta^{\frac{1}{p-1}})\}, 
\end{align}
see Figure \ref{figalbe}. 

    \begin{figure}[!htb]
    \centering 
    \includegraphics[scale=0.25]{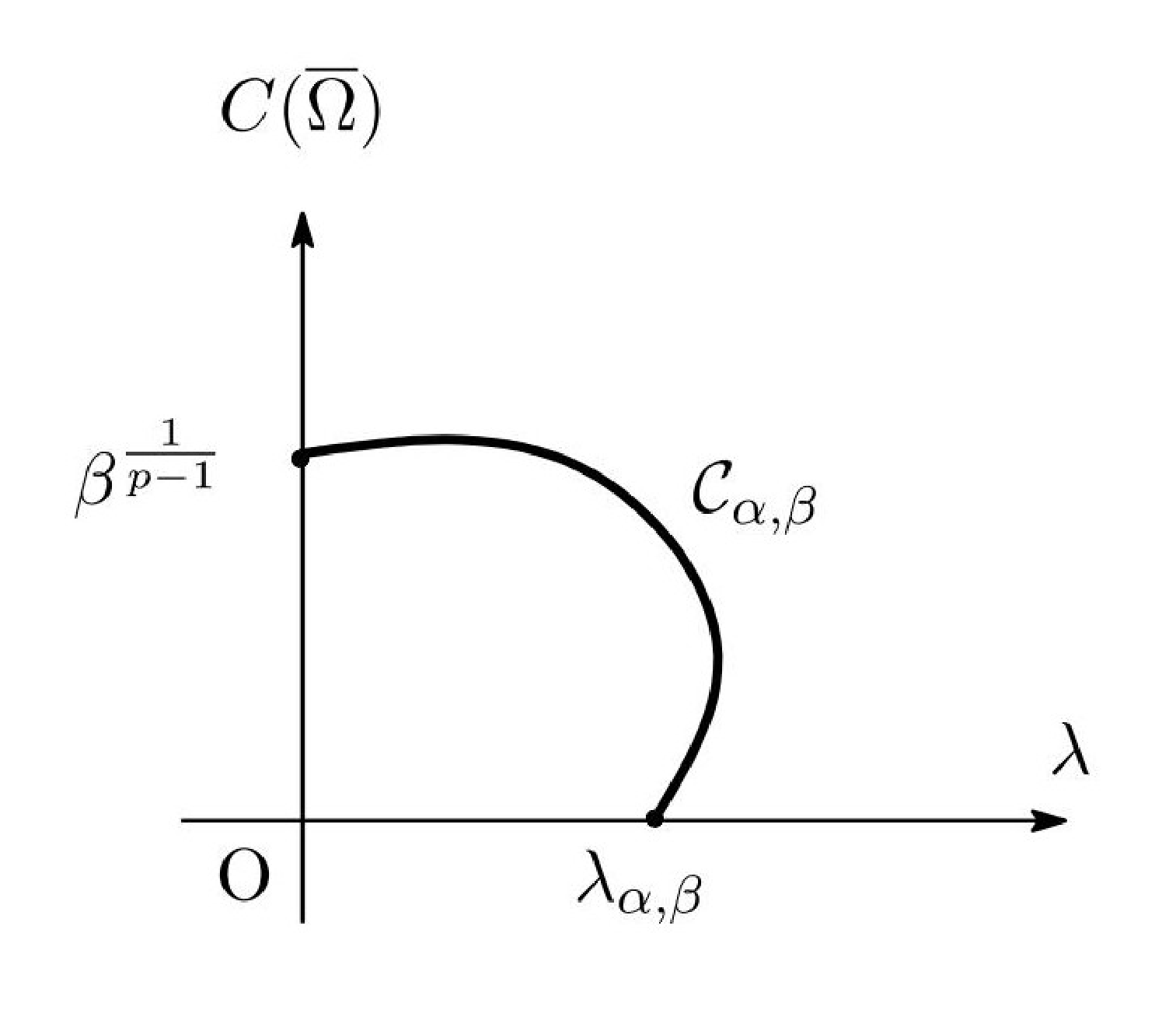} 
		  \caption{Admissible case for $\mathcal{C}_{\alpha, \beta}$.}
		\label{figalbe} 
    \end{figure} 

Then, we investigate the limiting behavior of the bounded component $\mathcal{C}_{\alpha, \beta}$ as $\alpha \to 0^{+}$. 
To this end, we employ a topological method proposed by Whyburn \cite{Wh64}, which reads as follows: Let $X$ be a metric space, and let $E_n \subset X$. Let 
\begin{align*}
& \varliminf_{n\to \infty} E_n := \{ x \in X : \lim_{n\to \infty}{\rm dist}\, (x, E_n) = 0 \}, \\
& \varlimsup_{n\to \infty} E_n := \{ x \in X : \varliminf_{n\to \infty}{\rm dist}\, (x, E_n) = 0 \}, 
\end{align*}
and then, we have (\cite[(9.12) Theorem]{Wh64}) 
\begin{theorem}[Whyburn]    \label{thm:W}
Assume that $\left\{ E_{n} \right\}_{n}$ is a sequence of connected sets which satisfies that 
\begin{enumerate} \setlength{\itemsep}{3pt} 
\item[(i)] $\displaystyle \bigcup_{n\geq 1}E_n$ is precompact,  
\item[(ii)] $\displaystyle \varliminf_{n\to \infty} E_n \neq \emptyset$. 
\end{enumerate}
Then, $\displaystyle \varlimsup_{n\to \infty} E_n$ is nonempty, closed and connected. 
\end{theorem}

We know from Propositions \ref{lem:albet:bddnorm} and \ref{153319}(i) that
\begin{align} \label{2021a}
\bigcup_{0<\alpha \leq \alpha_0} \mathcal{C}_{\alpha, \beta} \subset 
\{ (\lambda, u) \in [0, \infty)\times C(\overline{\Omega}) : 0\leq \lambda \leq \Lambda_\beta, \ 0\leq u \leq 1 \}. 
\end{align}
Let $\alpha_n \to 0^{+}$. Then, by a standard bootstrap argument attributed to elliptic regularity, we can deduce from \eqref{2021a} that $\bigcup_{n\geq1} \mathcal{C}_{\alpha_n, \beta}$ is precompact. In addition, from \eqref{221725}, we derive that 
\begin{align} \label{2021b}
(0, \beta^{\frac{1}{p-1}}) \in \varliminf_{n} \mathcal{C}_{\alpha_n, \beta}. 
\end{align}
Thus, Theorem \ref{thm:W} applies to $\{ \mathcal{C}_{\alpha_n, \beta}\}_n$, and 
we deduce that 
\[ 
\mathcal{C}_{\beta}:= \varlimsup_{n} \mathcal{
C}_{\alpha_n, \beta}
\]
is bounded, closed, and connected in $[0, \infty)\times C(\overline{\Omega})$. 
Moreover, $(0,0), (0, \beta^{\frac{1}{p-1}})\in \mathcal{C}_\beta$ because of \eqref{2026s} and \eqref{2021b}. 
As a matter of fact, $\mathcal{C}_{\beta}$ joins $(0,0)$ to $(0, \beta^{\frac{1}{p-1}})$. 

We claim that $u$ is a nonnegative solution of \eqref{p-al0} if $(\lambda, u) \in \mathcal{C}_ \beta$. Indeed, by the definition of $\mathcal{C}_{\beta}$, for $(\lambda, u) \in \mathcal{C}_ \beta$ there may exist $(\lambda_n, u_n) \in \mathcal{C}_{\alpha_n, \beta}$ with $\alpha_n \to 0^{+}$ such that $(\lambda_n, u_n) \rightarrow (\lambda, u)$ in $[0, \infty)\times C(\overline{\Omega})$, and then,  
\begin{align} \label{wfaln}
\int_\Omega \biggl( \nabla u_n \nabla \varphi - \beta u_n \varphi + u_n^p \varphi \biggr) + \lambda_n \int_{\partial \Omega} (u_n +\alpha_n)^{q-1}u_n \varphi = 0, \quad \varphi \in H^1(\Omega). 
\end{align}
By substituting $\varphi = u_{n}$ into \eqref{wfaln}, it is observed that $u_n$ is bounded in $H^1(\Omega)$; thus, up to a subsequence, $u_n \rightharpoonup u$, $u_n \rightarrow u$ in $L^2(\Omega)$ and $L^2(\partial \Omega)$, and $u_n \rightarrow u$ a.e.\ in $\Omega$ and on $\partial\Omega$, and thus, $u\in H^1(\Omega)$. 
Taking the limit, we deduce by the Lebesgue dominated convergence theorem that
$\int_\Omega u_n^p \varphi \rightarrow \int_\Omega u^p \varphi$; thus, 
\begin{align*}
\int_\Omega \biggl( \nabla u_n \nabla \varphi - \beta u_n \varphi + u_n^p \varphi \biggr)  \longrightarrow 
\int_\Omega \biggl( \nabla u \nabla \varphi - \beta u \varphi + u^p \varphi \biggr). 
\end{align*}
A similar argument is carried out on $\partial \Omega$. Since 
\begin{align*}
    (u_n +\alpha_n)^{q-1}u_n = \biggl( \frac{u_n}{u_n + \alpha_n} \biggr)^{1-q} u_n^q 
    \leq 1, 
\end{align*}
and 
\begin{align*}
 &   \biggl( \frac{u_n}{u_n + \alpha_n} \biggr)^{1-q} u_n^q \longrightarrow u^q \quad \mbox{  for } \ x \in \partial\Omega \ \ \mbox{ satisfying that } \ u(x)>0, \\ 
 & 0\leq \biggl( \frac{u_n}{u_n + \alpha_n} \biggr)^{1-q} u_n^q \leq u_n^q \longrightarrow 0 \quad \mbox{ for } \ x\in \partial\Omega \ \ \mbox{ satisfying that } \ u(x)=0, 
 \end{align*}
we use the Lebesgue dominated convergence theorem again to obtain 
\begin{align*}
    \lambda_n \int_{\partial \Omega} (u_n +\alpha_n)^{q-1}u_n \varphi \longrightarrow \lambda \int_{\partial\Omega} u^q \varphi. 
\end{align*}
Thus, taking the limit in \eqref{wfaln}, 
\begin{align*}
    \int_\Omega \biggl( \nabla u \nabla \varphi - \beta u \varphi + u^p \varphi \biggr) + \lambda \int_{\partial \Omega} u^q \varphi = 0,  
\end{align*}
as desired. 

Then, we claim that $\mathcal{C}_\beta \setminus \{ (0,0), (0, \beta^{\frac{1}{p-1}})\}$ consists of the positive solutions of \eqref{p-al0} with $\lambda > 0$. 
Indeed, Proposition \ref{215925} provides   
\begin{align}  \label{172531}
\mathcal{C}_\beta \cap \{ (0,u) : u\geq0 \} \} = \left\{ (0,0), \, (0, \beta^{\frac{1}{p-1}}) \right\}, 
\end{align}
and Proposition \ref{222825} provides 
\begin{align*}
\mathcal{C}_\beta \cap \{ (\lambda, 0) : \lambda \geq 0 \} = \left\{ (0,0) \right\}. 
\end{align*}
see Figure \ref{figbeta}. The desired claim follows.

    \begin{figure}[!htb]
    \centering 
    \includegraphics[scale=0.25]{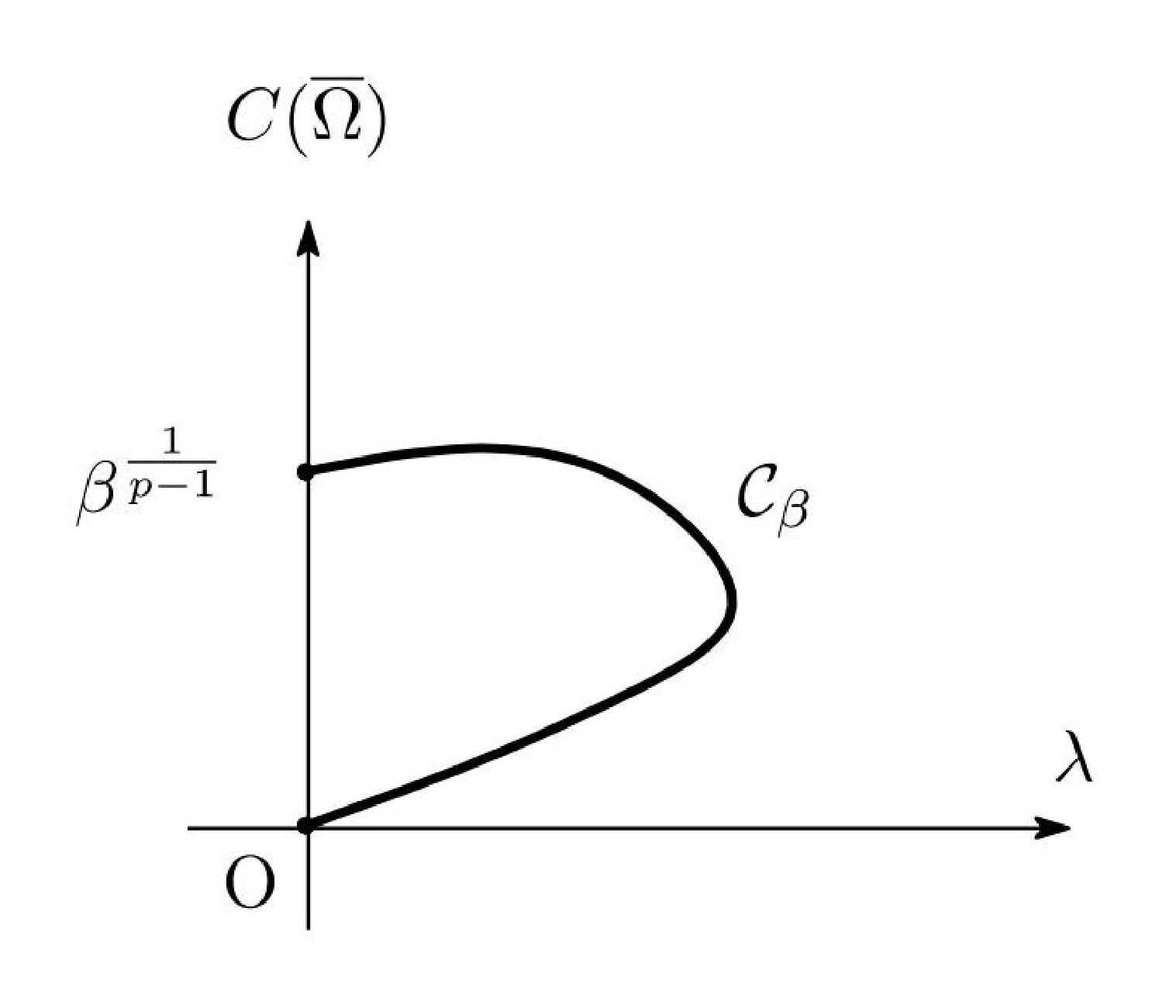} 
	  \caption{Admissible case for $\mathcal{C}_{\beta}$.}
		\label{figbeta} 
    \end{figure} 

Then, additionally assuming $pq\leq1$, we evaluate the limiting behavior of $\mathcal{C}_\beta$ as $\beta \to 1^{-}$ by employing the same approach as that of $\mathcal{C}_{\alpha,\beta}$ as $\alpha \to 0^{+}$. 
If $pq\leq1$, then we know from Propositions \ref{lem:albet:bddnorm} and \ref{153319}(ii) that $\mathcal{C}_\beta$ is bounded in $[0, \infty)\times C(\overline{\Omega})$, uniformly in $\beta\in [\beta_{0},1]$, as $\beta \to 1^{-}$; thus, 
\begin{align*}
\bigcup_{\beta_0\leq \beta < 1}\mathcal{C}_\beta 
\subset \{ (\lambda, u) \in [0, \infty)\times C(\overline{\Omega}) : 0\leq \lambda \leq \Lambda_0, \ 0\leq u \leq 1 \}. 
\end{align*}
For $\beta_n \to 1^{-}$, we see in the same manner that 
$\bigcup_{n\geq1}\mathcal{C}_{\beta_{n}}$ is precompact. Since $(0,0)\in \mathcal{C}_{\beta_{n}}$, and since $(0,\beta_n^{1/(p-1)})\in \mathcal{C}_{\beta_n}$ from \eqref{172531}, we have $(0,0), (0,1) \in \varliminf_n \mathcal{C}_{\beta_n}$. Theorem \ref{thm:W} now applies, and then, 
\[
\mathcal{C}^{\ast}_{p,q}:= \varlimsup_n \mathcal{C}_{\beta_n}
\]
is bounded, closed and connected in $[0, \infty)\times C(\overline{\Omega})$ and joins $(0,0)$ to $(0,1)$. In addition, it is seen similarly as above that $\mathcal{C}^{\ast}_{p,q}$ consists 
of the nonnegative solutions of \eqref{p}. 
Proposition \ref{215925} shows  
\begin{align} \label{Cpq919}
    \mathcal{C}^{\ast}_{p,q} \cap \{ (0, u) : u\geq 0\} = \{ (0,0), (0,1)\}. 
\end{align}

\vspace{3pt} 

\begin{proof}[End of Proofs for assertions (II) and (III) of Theorem \ref{mt}] 
We evaluate the case $pq<1$ and end the proof of assertion (II). 
In this case,  Assertion \eqref{171325} shows 
\begin{align} \label{Cpq916}
\mathcal{C}^{\ast}_{p,q} \cap \{ (\lambda, 0) : \lambda > 0\} = \emptyset. 
\end{align}
We have verified that $\mathcal{C}_{p,q}^{\ast}$ is connected in $[0,\infty)\times C(\overline{\Omega})$; therefore, combining \eqref{Cpq919}, \eqref{Cpq916},  Proposition \ref{215925}, and Remark \ref{rem1030}(i) provides us with the existence of the positive solutions $U_{1}$ and $U_{2}$ claimed in item (II). 
The nonexistence assertion follows from Proposition \ref{153319}(ii). Assertion (II-i) follows from Proposition \ref{215925} and \eqref{171325}. 
Regarding assertion (II-ii), the existence assertion for $\mathcal{C}_{p,q}^{\ast}$ follows from \eqref{Cpq919}. The rest is evaluated by assertion (II-i). 
The proof of assertion (II) of Theorem \ref{mt} is complete.

Then, we evaluate the case $pq=1$ and end the proof of assertion (III). The unique positive solution $U_{1}$ claimed in item (III) is obtained by the combination of \eqref{Cpq919}, \eqref{171325b},  Proposition \ref{215925}, and Remark \ref{rem1030}(i). 
The nonexistence assertion follows from Proposition \ref{153319}(ii). 
Assertion (III-i) follows from \eqref{171325b} and Proposition \ref{215925}. 
Regarding assertion (III-ii), the existence assertion for $\mathcal{C}_{p,q}^{\ast}$ follows from \eqref{Cpq919}. The rest is deduced from assertion (III-i). The proof of assertion (III) of Theorem \ref{mt} is now complete. 
\end{proof}

\section{Unbounded positive solution sets: case $\beta_\Omega = 1$ and $pq>1$} 

\label{sec:sub}

In the final section, we consider the case where $\beta_\Omega = 1$ and $pq>1$ 
and prove assertion (I) of Theorem \ref{mt}. We exploit the usual sub-- and supersolution method \cite[(2.1) Theorem]{Am76N} to verify the existence part. 
A function $u \in C^{2+\theta}(\overline{\Omega})$, $0<\theta <1$, satisfying that $u>0$ in $\overline{\Omega}$ is said to be a {\it subsolution} of \eqref{p} if the following condition holds. 
\begin{align*} 
\begin{cases}
-\Delta u \leq u - u^{p} & \mbox{ in } \Omega, \\
\frac{\partial u}{\partial \nu}  \leq -\lambda u^{q} & \mbox{ on } \partial\Omega. 
\end{cases}  
\end{align*}
A {\it supersolution} of \eqref{p} is defined by reversing the inequalities.

It is clear that $\psi = 1$ is a supersolution of \eqref{p} for every $\lambda > 0$. 
Then, we construct a smaller subsolution of \eqref{p} than $\psi = 1$. 
For $\varepsilon, \tau>0$, set 
\[
\phi_\varepsilon (x) = \varepsilon (\phi_\Omega (x) + \varepsilon^\tau), \quad x \in \overline{\Omega}, 
\]
and then, we have the following. Although the proof of Lemma \ref{prop:sub} has been already evaluated in \cite[Lemma 3.1]{Um2024}, we present it for the sake of the completeness of our arguments in this paper.

\begin{lem} \label{prop:sub}
Assume that $\beta_\Omega = 1$ and $pq>1$. 
Let $\frac{1-q}{q} < \tau < p-1$. 
Then, for $\Lambda > 0$ there exists $\overline{\varepsilon}=\overline{\varepsilon}(\tau, \Lambda)>0$ such that $\phi_{\varepsilon}$ with 
$\varepsilon \in (0, \overline{\varepsilon}]$ is a subsolution of \eqref{p} for  $\lambda \in [0, \Lambda]$. 
Furthermore, $u \geq \phi_{\overline{\varepsilon}}$ in $\overline{\Omega}$ for a positive solution $u > 0$ in $\overline{\Omega}$ of \eqref{p} with $\lambda \in [0,\Lambda]$. Here, $\overline{\varepsilon}$ does not depend on $\lambda \in [0,\Lambda]$. 
\end{lem}
\begin{proof}
For the former assertion, take $0<\varepsilon\leq 1$; then, we use the condition $p-\tau -1>0$ to deduce that  
\begin{align*}
-\Delta \phi_\varepsilon -\phi_\varepsilon +\phi_{\varepsilon}^{p} 
\leq \varepsilon^{1+\tau}\biggl\{ -1 + \varepsilon^{p-\tau -1} \left( 1+\max_{\overline{\Omega}}\phi_\Omega \right)^p \biggr\}<0 
\quad\mbox{ in } \Omega  
\end{align*}
if 
\[
0<\varepsilon \leq \overline{\varepsilon}_{1} < 
\min 
\left( 
1, \;  \left( 
			\frac{1}{(1+\max_{\overline{\Omega}} \phi_{\Omega})^{p}} 	
				\right)^{\frac{1}{p-\tau-1}}  
\right). 
\]
For $\Lambda > 0$, we use \eqref{bdry1513} and the condition $\tau > \frac{1-q}{q}$ to deduce that 
\begin{align*}
\frac{\partial \phi_\varepsilon}{\partial \nu} + \lambda \phi_\varepsilon^q 
= \varepsilon \frac{\partial \phi_{\Omega}}{\partial \nu} + \lambda \varepsilon^{(1+\tau)q} 
\leq \varepsilon (-c_1 + \Lambda \varepsilon^{q+\tau q -1}) < 0 \quad 
\mbox{ on } \partial\Omega 
\end{align*}
if $\lambda\in [0,\Lambda]$ and 
\[
0<\varepsilon \leq \overline{\varepsilon}_{2} < \left( 
\frac{c_{1}}{\Lambda} \right)^{\frac{1}{q+\tau q-1}},   
\]
where 
$c_1 = \min_{\partial\Omega}(-\frac{\partial \phi_{\Omega} }{\partial \nu})>0$. Therefore, $\overline{\varepsilon} = \min (\overline{\varepsilon}_{1}, \overline{\varepsilon}_{2})$ is the desired constant.

For the latter assertion, we assume to the contrary that 
$u\not\geq \phi_{\overline{\varepsilon}}$ in $\overline{\Omega}$ for a positive solution $u>0$ in $\overline{\Omega}$ of \eqref{p} 
with $\lambda \in [0, \Lambda]$. Because $\varepsilon \mapsto \phi_\varepsilon$ is increasing and 
$\| \phi_\varepsilon \|_{C(\overline{\Omega})} \rightarrow 0$ as $\varepsilon \to 0^{+}$, we can take $\varepsilon_1\in (0, \overline{\varepsilon})$ such that 
\begin{align}
\begin{cases}
u\geq \phi_{\varepsilon_1} &  \ \ \mbox{ in } \ \overline{\Omega}, \\
u(x_1) = \phi_{\varepsilon_1}(x_1) & \ \ \mbox{ for some } \ x_1 \in \overline{\Omega}. 
\end{cases} \label{bdry1251}
\end{align}
Take a small $c>0$ such that $u, \phi_{\varepsilon_1}\geq c$ in $\overline{\Omega}$, and then, 
\[
\frac{u^{q}-\phi_{\varepsilon_{1}}^{q}}{u-\phi_{\varepsilon_{1}}}\leq q c^{q-1} \quad\mbox{ if } \ u>\phi_{\varepsilon_{1}}. 
\]
We choose a sufficiently large $K>0$ such that 
$f_K(t)=Kt + t-t^p$ is increasing for $t \in \left[ 0, \, \| u \|_{C(\overline{\Omega})} \right]$ and a sufficiently large $M>0$ such that $M-\Lambda q c^{q-1}>0$. 
Since $\phi_{\varepsilon_1}$ is a subsolution (not a positive solution) 
of \eqref{p}, we deduce that 
\begin{align*} 
(-\Delta + K)(u-\phi_{\varepsilon_1})> f_K(u) - f_K(\phi_{\varepsilon_1}) \geq 0 
\quad\mbox{ in } \ \Omega,
\end{align*}
and for $x \in \partial\Omega$ satisfying $u>\phi_{\varepsilon_1}$, 
\begin{align*}
\biggl( \frac{\partial}{\partial \nu} + M \biggr) (u-\phi_{\varepsilon_1})
&> -\lambda u^q + \lambda \phi_{\varepsilon_1}^q + M(u-\phi_{\varepsilon_1})  \\ 
&= \biggl( M - \lambda \frac{u^q - \phi_{\varepsilon_1}^q}{u - \phi_{\varepsilon_1}} \biggr)(u-\phi_{\varepsilon_1}) \\
& \geq (M - \Lambda q c^{q-1})(u-\phi_{\varepsilon_1})>0. 
\end{align*}
Hence, applying the strong maximum principle and boundary point lemma 
provides $u-\phi_{\varepsilon_1}>0$ in $\overline{\Omega}$, which contradicts \eqref{bdry1251}. 
\end{proof} 

\begin{proof}[Proof of assertion (I) of Theorem \ref{mt}] 
First, we verify the existence part. Let $\lambda>0$. We may assume that $\phi_{\overline{\varepsilon}}\leq1$ in $\overline{\Omega}$ in Lemma \ref{prop:sub}. 
The sub-- and supersolution method \cite[(2.1) Theorem]{Am76N} applies, and then, 
problem \eqref{p} possesses 
a minimal positive solution $\underline{u}$ and a maximal positive solution $\overline{u}$ in the order interval $[\phi_{\overline{\varepsilon}},1]$, meaning that 
$\underline{u}\leq u \leq \overline{u}$ in $\overline{\Omega}$ for a positive solution $u$ of \eqref{p} meeting the condition that 
$\phi_{\overline{\varepsilon}}\leq u \leq 1$ in $\overline{\Omega}$. 
As a matter of fact, Proposition \ref{lem:albet:bddnorm} and Lemma \ref{prop:sub} show that 
$\phi_{\overline{\varepsilon}}\leq u<1$ in $\overline{\Omega}$ for a positive solution $u>0$ in $\overline{\Omega}$ of \eqref{p}; therefore, $\underline{U}_{1}=\underline{u}$ and $\overline{U}_{1}=\overline{u}$ are as desired. 

Then, the combination of \eqref{171325c}, Proposition \ref{215925}, and Remark \ref{rem1030}(i) provides assertions (I-i) and (I-ii). 
Assertion (I-iii) follows from Propositions \ref{prop:asympt} and \ref{prop:unprofi}. The proof of assertion (I) of Theorem \ref{mt} is now complete. 
\end{proof}

\appendix 

\section{Asymptotic property for Dirichlet logistic problem} 

The readers may be interested if assertion (I-iii) of Theorem \ref{mt} is consistent 
with \eqref{convtouD}.  To discuss this issue, we employ a domain perturbation of $\Omega$ for \eqref{Dp} in the case of $\beta_\Omega = 1$.

Let $\Omega_k \subset \mathbb{R}^N$, $k=1,2,3,\ldots$, be a bounded domain with a smooth boundary $\partial \Omega_k$ that satisfies the conditions. 
\begin{align} \label{Omn} 
\begin{cases}
& \mbox{$\bullet$ there exists a ball $B_0 \subset \mathbb{R}^N$ such that $\Omega \Subset \Omega_k \Subset B_0$ for every $k$, } \\
& \mbox{\quad which means that $\overline{\Omega} \subset \Omega_k$ and $\overline{\Omega_k}\subset B_0$ for every $k$,  } \\ 
& \mbox{$\bullet$ for any open subset $D\supset \overline{\Omega}$, there exists $k_0$ such that $\overline{\Omega_k} \subset D$ for all $k\geq k_0$. } 
\end{cases}     
 \end{align}
We consider the scenario that $\Omega_k \rightarrow \Omega$ is defined in the sense of \eqref{Omn}, and we present $\Omega_k = \{ x : {\rm dist}(x, \Omega)<\frac{1}{k} \}$ as an example satisfying \eqref{Omn}.

By the monotonicity of $\beta_\Omega$ with respect to $\Omega$, we find that $\beta_{\Omega_k}<1$, and then, by denoting by $u_{\mathcal{D}}^{\Omega_k}$ the unique positive solution of \eqref{Dp} for $\Omega$ replaced by $\Omega_k$, we can prove that as $k\to \infty$, 
\begin{align}
& u_{\mathcal{D}}^{\Omega_k} \longrightarrow 0 \quad\mbox{ in } H^1(\Omega), \label{Omk1}\\ 
& \frac{u_{\mathcal{D}}^{\Omega_k}}{\| u_{\mathcal{D}}^{\Omega_k} \|_{\Omega_k}} \longrightarrow \phi_\Omega \quad\mbox{ in } H^1(\Omega). \label{Omk2}
\end{align}
From \eqref{convtouD}, it is easy to deduce that, for a fixed $k$, 
\begin{align} \label{Omk3}
\frac{u_{\lambda}^{\Omega_k}}{\| u_\lambda^{\Omega_k} \|_{\Omega_k}} \longrightarrow \frac{u_{\mathcal{D}}^{\Omega_k}}{\| 
u_{\mathcal{D}}^{\Omega_k} \|_{\Omega_k}} \quad \mbox{ in } H^1(\Omega_k) \quad\mbox{ as } \lambda \to \infty, 
\end{align}
where $u_{\lambda}^{\Omega_{k}}$ denotes the unique positive solution $u_{\lambda}$ of 
\eqref{p} for $\Omega$ replaced by $\Omega_{k }$. 
Considering \eqref{Omk3}, assertions \eqref{Omk1} and \eqref{Omk2} indicate that \eqref{convtouD} and assertion (I-iii) of Theorem \ref{mt} are consistent. 
Although \eqref{Omk1} and \eqref{Omk2} may be known, we provide simple proofs for them (cf.\ \cite{Daners2008}). 

Then, the unique positive solution $u_{\mathcal{D}}^{\Omega_k} \in C^{2+\theta}(\overline{\Omega_k}) \cap H^1_0(\Omega_k)$ is extended by $0$ to $B_0$ as a function in $H^1_0(B_0)$, which is still denoted by the same notation; then, $\| u_{\mathcal{D}}^{\Omega_k} \|_{B_0}=\| u_{\mathcal{D}}^{\Omega_k} \|_{\Omega_k}$. Hereafter, we write $u_{\mathcal{D}}^{\Omega_k}$ by $u_k$.

\begin{proof}[Proof of \eqref{Omk1}]
First, we claim that $u_k$ is bounded in $H^1_0(B_0)$. 
Proposition \ref{lem:albet:bddnorm} shows that $u_k \leq 1$ in $\Omega_k$. Hence, we infer that 
\begin{align*}
\int_{\Omega_k} |\nabla u_k|^2 = \int_{\Omega_k} \biggl( u_k^2 - u_k^{p+1} \biggr) \leq c, 
\end{align*}
as desired. Then, up to a subsequence, 
$u_k \rightharpoonup u_\infty$ , $u_k \rightarrow u_\infty$ in $L^2(B_0)$, and $u_k \rightarrow u_{\infty}$ a.e.\ in $B_0$ for some $u_\infty \in H^1_0(B_0)$ nonnegative. 

Next, we claim that $u_\infty \in H^1_0(\Omega)$. Given $x \in B_0\setminus \overline{\Omega}$, from \eqref{Omn}, there exists $k_0$ such that if $k \geq k_0$, then $u_k(x)=0$. Therefore, $u_\infty (x) = 0$ in $B_0\setminus \overline{\Omega}$, and the desired assertion is deduced by using  \cite[Proposition 5.4.3]{Daners2008}. 

By definition, 
\begin{align*}
\int_{\Omega_k} \biggl( \nabla u_k \nabla \varphi - u_k \varphi + u_k^p \varphi \biggr) = 0, \quad \forall \varphi \in H^1_0(\Omega_k).
\end{align*}
A function $\varphi \in H^1_0(\Omega)$ is extended by $0$ to $B_0$ as a function in $H^1_0(B_0)$, and then, for any $\varphi \in H^1_0(\Omega)$, 
\begin{align*}
\int_{\Omega_k} \biggl( \nabla u_k \nabla \varphi - u_k \varphi + u_k^p \varphi \biggr) = 0. 
\end{align*}
Thus, 
\begin{align*}
\int_{B_0} \biggl( \nabla u_k \nabla \varphi - u_k \varphi + u_k^p \varphi \biggr) = 0. 
\end{align*}
Taking the limit provides 
\begin{align*}
\int_{B_0} \biggl( \nabla u_\infty \nabla \varphi - u_\infty \varphi + u_\infty^p \varphi \biggr) = 0,      
\end{align*}
where we used the Lebesgue dominated convergence theorem to deduce that $\int_{B_0} u_k^p \varphi \rightarrow \int_{B_0} u_\infty^p \varphi$. 
Then, for $u_{\infty} \in H^{1}_{0}(\Omega)$, we get 
\begin{align*}
\int_{\Omega} \biggl( \nabla u_\infty \nabla \varphi - u_\infty \varphi + u_\infty^p \varphi \biggr) = 0, \quad \forall \varphi \in H^1_0(\Omega).     
\end{align*}
This means that $u_\infty$ is a nonnegative weak solution of \eqref{Dp}, and as a matter of fact, $u_{\infty} = 0$ in $\Omega$ because $\beta_\Omega = 1$. 
Thus, $u_\infty = 0$ in $B_0$, and $\int_{B_0} u_k^2 \rightarrow 0$.  

To our end, it suffices to verify that $\| u_k \|_{B_0} \rightarrow 0$. 
By the weak lower semi--continuity, 
\begin{align*}
0=\int_{B_0}\biggl( |\nabla u_\infty|^2 - u_\infty^2 \biggr) 
&\leq \varliminf_{k\to \infty} \int_{B_0} \biggl( |\nabla u_k|^2 - u_k^2 \biggr) \\ 
&\leq \varlimsup_{k\to \infty} \int_{B_0} \biggl( |\nabla u_k|^2 - u_k^2 \biggr) \\
&= \varlimsup_{k\to \infty} \biggl( -\int_{B_0} u_k^{p+1} \biggr) \leq 0. 
\end{align*}
Hence, $\int_{B_0} \left( |\nabla u_k|^2 - u_k^2 \right) \rightarrow 0$. Because $\int_{B_0}u_k^2 \rightarrow 0$, we deduce that $\int_{B_0}|\nabla u_k|^2 \rightarrow 0$, as desired. 
\end{proof}

Then, we extend $\phi_\Omega \in C^{2+\theta}(\overline{\Omega}) \cap H^1_0(\Omega)$ by $0$ to $B_0$ as a function in $H^1_0(B_0)$, satisfying $\| \phi_\Omega\|_{B_0}=1$. 

\begin{proof}[Proof of \eqref{Omk2}] 
For $w_k = \frac{u_k}{\| u_k\|_{\Omega_k}}\in H^{1}_{0}(B_{0})$, $\| w_k \|_{B_0}=1$ because $\| u_k \|_{\Omega_k} = \| u_k \|_{B_0}$. 
Up to a subsequence, $w_k \rightharpoonup w_\infty \geq0$, 
$w_k \rightarrow w_\infty$ in $L^2(B_0)$, and $w_{\infty} \rightarrow w_{\infty}$ a.e.\ in $B_{0}$ for some $w_\infty \in H^1_0(B_0)$. Similarly as for $u_\infty$, we deduce that $w_\infty = 0$ in $B_0\setminus \overline{\Omega}$; thus, $w_\infty \in H^1_0(\Omega)$. 

Since $\beta_\Omega = 1$, the assertion $w_\infty \in H^1_0(\Omega)$ implies that 
\begin{align}    
0 \leq \int_\Omega \biggl( |\nabla w_\infty|^2 - w_\infty^2 \biggr)  
& = \int_{B_0} \biggl( |\nabla w_\infty|^2 - w_\infty^2 \biggr) \nonumber \\
&\leq \varliminf_{k\to \infty} \int_{B_0} \biggl( |\nabla w_k|^2 - w_k^2 \biggr) \nonumber  \\ 
&\leq \varlimsup_{k\to \infty} \int_{B_0} \biggl( |\nabla w_k|^2 - w_k^2 \biggr) \nonumber \\ 
& = \varlimsup_{k\to \infty} \biggl( -\int_{B_0} \| u_k \|_{\Omega_k}^{p-1}w_k^{p+1} \biggr) \leq 0.  \label{wn}
\end{align}
Hence, $\int_\Omega (|\nabla w_\infty|^2 
- w_\infty^2)=0$, and $w_{\infty} = s\phi_{\Omega}$ for some $s\geq 0$. 
Further, we deduce from \eqref{wn} that $\| w_{k}\|_{B_{0}} \rightarrow \| w_{\infty}\|_{B_{0}}$, and thus, $w_{k} \rightarrow w_{\infty}$ in $H^{1}_{0}(B_{0})$ 
since $w_{k} \rightharpoonup w_{\infty}$ in $B_{0}$. Therefore, 
$\| w_\infty\|_{B_0}=1$, and $s=1$, i.e., $w_{\infty} = \phi_{\Omega}$, as desired. 
\end{proof}


\end{document}